\documentclass[12pt]{amsart}

\input{prepictex}
\input{pictex}
\input{postpictex}

\newtheorem{thm}{Theorem}[section]

\newtheorem{lem}[thm]{Lemma}

\newtheorem{cor}[thm]{Corollary}

\newtheorem{prop}[thm]{Proposition}

\theoremstyle{definition}

\newtheorem{defn}[thm]{Definition}
\newtheorem{defns}[thm]{Definitions}

\newtheorem{remark}{Remark}

\newtheorem{remarks}{Remarks}

\newcounter{picture}

\newcommand{\FF}{{\Bbb F}}
\newcommand{\KK}{{\Bbb K}}
\newcommand{\NN}{{\Bbb N}}

\newcommand{\QQ}{{\Bbb Q}}
\newcommand{\RR}{{\Bbb R}}

\newcommand{\ZZ}{{\Bbb Z}}

\newcommand{\cA}{{\mathcal A}}

\newcommand{\cC}{{\mathcal C}}

\newcommand{\cE}{{\mathcal E}}
\newcommand{\cF}{{\mathcal F}}
\newcommand{\cG}{{\mathcal G}}

\newcommand{\cK}{{\mathcal K}}
\newcommand{\cL}{{\mathcal L}}

\newcommand{\cO}{{\mathcal O}}
\newcommand{\cP}{{\mathcal P}}

\newcommand{\cR}{{\mathcal R}}
\newcommand{\cS}{{\mathcal S}}
\newcommand{\cT}{{\mathcal T}}

\newcommand{\cW}{{\mathcal W}}

\newcommand{\D}{{\Delta}}
\newcommand{\G}{{\Gamma}}
\newcommand{\La}{{\Lambda}}
\newcommand{\Om}{{\Omega}}
\newcommand{\Si}{{\Sigma}}
\newcommand{\s}{{\sigma}}
\newcommand{\w}{{\omega}}
\newcommand{\wa}{{\varpi}}

\newsymbol\rtimes 226F 

\newcommand{\aut}{{\text{\rm Aut}}}
\newcommand{\traut}{{\text{\rm Aut}_{\text{\rm tr}}(\D)}}

\newcommand{\tps}{{\text{\rm{III}}}_{1/p^2}}
\newcommand{\tqs}{{\text{\rm{III}}}_{1/q^2}}
\newcommand{\btqs}{%
{\text{\rm{\bf{III}}}}_{{\text{\bf{1}}}/{\text{\bf{q}}}^{\text{\bf{2}}}}}
\newcommand{\pgl}{{\text{\rm{PGL}}}}
\newcommand{\gl}{{\text{\rm{GL}}}}

\begin{document}

\title{Triangle Buildings and  Actions of Type $\btqs$}
\author{Jacqui Ramagge and Guyan Robertson}
\address{Mathematics Department, University of Newcastle, Callaghan, NSW 2308,
Australia}
\email{guyan@maths.newcastle.edu.au}
\thanks{This research was funded by the Australian Research Council.}
\maketitle

\begin{abstract}
We study certain group actions on triangle buildings and their boundaries and
some von Neumann algebras which can be constructed from them.  In
particular, for
buildings of order $q\geq 3$ certain natural actions on the boundary are
hyperfinite of type $\tqs$.
\end{abstract}

\section{Introduction}

We begin with a triangle building $\D$ whose one-skeleton is the Cayley
graph of a
group $\G$, so that $\G$ acts simply-transitively on the vertices of $\D$
by left
multiplication.  The group $\G$ (often referred to as an $\tilde{A}_2$
group) has a
relatively simple combinatorial structure, yet also has Kazhdan's
property~(T)~\cite{CMS}. It is interesting to note that not all such groups
$\G$ can
be embedded naturally as lattices in $\pgl\left(3,\KK\right)$ where $\KK$ is a
local field~\cite[II \S8]{CMSZ}. As a tool in our studies we introduce the
notion of a
periodic apartment in $\D$ . We prove several useful facts about periodic
apartments and the periodic limit points they define on the boundary,
$\Om$, of $\D$. The boundary $\Om$ is a totally disconnected compact Hausdorff
space and there is a family of mutually absolutely continuous Borel probability
measures $\{\nu_v\}$ on $\Om$ indexed by the vertices of $\D$. We show that
periodic boundary points form a dense subset of the boundary of trivial measure
with respect to this class. We also prove that the only boundary points
stabilized
by the action of $\G$ are the periodic limit points. This enables us to
deduce that
the action of $\G$ on $\Om$ is free. By analogy with the tree case we show
that the
action of $\G$ on $\Om$ is also ergodic, thus establishing that
$L^\infty(\Om)\rtimes\G$ is a factor.  Since the action of $\G$ is amenable,
$L^\infty(\Om)\rtimes\G$ is a hyperfinite factor. In Theorem~\ref{main theorem
for Gamma} we prove that, if $q\geq 3$, $L^\infty(\Om)\rtimes\G$ is in fact the
hyperfinite factor of type~$\tqs$. The corresponding result for homogeneous
trees was proved  in~\cite{RR}.

In \S\ref{section 5}, we consider the action of some other groups $G$ acting on
triangle buildings and we show that, under certain conditions, the action is of
type~$\tqs$. In particular, Theorem~\ref{PGLFq factor} asserts that if
$\Om$ is the
boundary of the building associated to $\pgl\left(3,\FF_q((X))\right)$, then the
action of $\pgl\left(3,\FF_q(X)\right)$ is of type~$\tqs$.

The work on periodic apartments was initially motivated by some of the material
in~\cite{MSZ}, where the authors first met the notion of  doubly periodic
apartments. In~\cite{Moz}, S.~Mozes had previously proved the abundance of
doubly periodic apartments in a different context. He used them to obtain deep
results on dynamical systems. We generalize the notion of doubly periodic
apartments introduced in that paper to that of periodic apartments and prove
some surprising properties these generalizations possess. The proof of
Theorem~\ref{existence} appeared in~\cite{MSZ}, but is included for the
insight it
provides, and  Corollary~\ref{cor 2} appeared as a lemma, without reference to a
result analogous to Lemma~\ref{rigidity}.

\subsection{Triangle Buildings}

We refer the reader to~\cite{Brown} and~\cite{Ronan} for more details of the
following facts on buildings.

A {\bf triangle building} is a thick affine building $\D$ of
type~$\tilde{A}_2$. Thus
$\D$ is a simplicial complex of rank 2 and consists of vertices, edges and
triangles. We call the maximal simplices, i.e.\ the triangles, {\bf
chambers}. The
vertices in $\D$ have one of three types and every chamber contains
precisely one
vertex of each type. Thus there is a type map $\tau$ defined on the vertices of
$\D$ such that $\tau(v)\in\ZZ/3\ZZ$ for each vertex $v\in\D$. An automorphism
$g$ of $\D$ is said to be {\bf type-rotating} if there exists $i\in\{
0,1,2\}$ such that
$\tau(gv)=\tau(v)+i$ for all vertices $v\in\D$, and is said to be {\bf
type-preserving} if $i=0$. We denote by $\traut$ the group of all
type-rotating automorphisms of $\D$. We shall assume that each edge lies on a
finite number of chambers, in which case each edge lies on precisely
$q+1$ chambers for some integer $q\geq 2$ and we call $q$ the {\bf order}
of $\D$.
There is a well-defined
$W$-valued distance function on the simplices of $\D$ where $W$ is the affine
Coxeter group of type~$\tilde{A}_2$. The length function $|\cdot |$ on $W$
enables
us to define an integer-valued distance function $d$ on~$\D$. Given a simplex
$\sigma\in\D$, the {\bf residue}, or {\bf star} determined by $\sigma$ is
the set of
all simplices $\rho$ satisfying $d(\sigma,\rho)=1$.

An {\bf apartment} $\cA$ in $\D$ is a subcomplex isomorphic to the Coxeter
complex of type~$\tilde{A}_2$. Hence each apartment $\cA$ can be realized as a
Euclidean plane tesselated by equilateral triangles. A {\bf root} in $\D$
is a root
in any apartment $\cA$ of $\D$. That is to say, a root is a half plane in a
geometric
realization of $\cA$. A {\bf wall} in $\D$ is a wall in any apartment $\cA$
and a {\bf
sector} is a simplicial cone in some apartment. Two sectors are said to be {\bf
equivalent}, or {\bf parallel}, if they contain a common subsector.

Given any vertex $v$ and any chamber $C$ in a common apartment $\cA$, their
convex hull in $\cA$ must in fact be contained in every apartment which contains
both $v$ and $C$. We note that given a vertex $v\in\cA$ we can decompose $\cA$
into six sectors emanating from, or {\bf based at}, $v$. Consider two such
sectors
$\cS_1, \cS_2$ which are opposite each other, as in Figure~\ref{opposite
sectors}.
The convex hull of $\cS_1$ and $\cS_2$ is all of $\cA$, so that $\cS_1$ and
$\cS_2$
uniquely determine~$\cA$.

\refstepcounter{picture}
\begin{figure}[htbp]
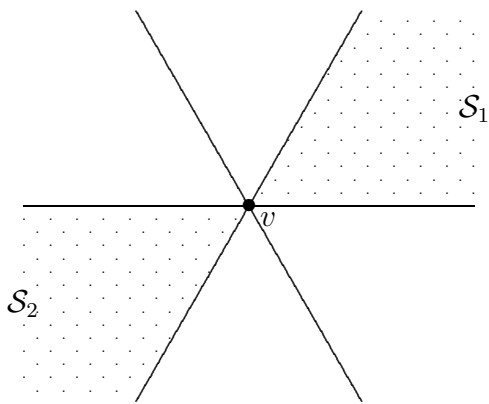
\label{opposite sectors}
{}\hfill
\beginpicture
\setcoordinatesystem units <0.5cm,0.866cm>    
\setplotarea x from -6 to 6, y from -4 to 4         
\put {$\bullet$} at 0 0
\put {$v$}             [t l]  at    0.3   -0.1
\put {$\cS_1$}   at    6 1.5
\put {$\cS_2$}  at    -6  -1.5
\putrule from -6    0     to  6    0
\setlinear
\plot -3  3   3 -3 /
\plot -3 -3   3  3 /
\vshade -6 -3 0 <z,z,z,z> -3 -3 0 <z,z,z,z> 0 0 0 <z,z,z,z> 3 0 3
<z,z,z,z> 6 0 3 /
\endpicture
\hfill{}
\caption{Opposite sectors in an apartment.}
\end{figure}

\begin{remark}\label{uncountable apartments}
There are an uncountable number of apartments containing any given vertex
$v\in\D$. To see this, suppose that $\cR$ is any root containing~$v$.
Without loss
of generality, suppose $v$ is on the outside wall $\cW$ of $\cR$ as in
Figure~\ref{outside wall}.

\refstepcounter{picture}
\begin{figure}[htbp]
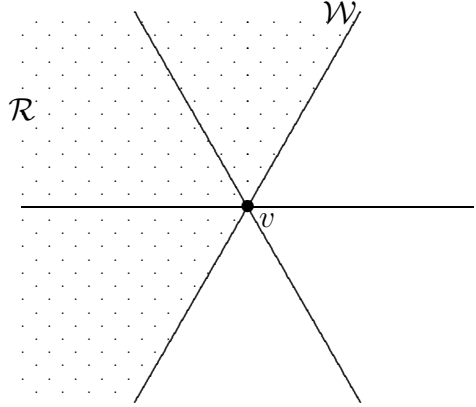
\label{outside wall}
{}\hfill
\beginpicture
\setcoordinatesystem units <0.5cm,0.866cm>    
\setplotarea x from -6 to 6, y from -4 to 4         
\put {$\bullet$} at 0 0
\put {$\cW$}  [r]         at 2.9  3
\put {$\cR$} at -6  1.5
\put {$v$}             [t l]  at    0.3   -0.1
\putrule from -6    0     to  6    0
\setlinear
\plot -3  3   3 -3 /
\plot -3 -3   3  3 /
\vshade -6 -3 3 <z,z,z,z> -3 -3 3 <z,z,z,z> 0 0 3 <z,z,z,z> 3 3 3   /
\endpicture
\hfill{}
\caption{A root $\cR$ containing the vertex $v$.}
\end{figure}

Consider any edge on $\cW$, say $E_1$. By thickness there exist at least two
distinct chambers, $C_1$ and $C_2$, incident on $E_1$ which are not in $\cR$.
Adjoining either of these chambers will uniquely determine new roots $\cR_i$
which are the convex hulls of $C_i$ and $\cR$. Note that $\cR_1\neq\cR_2$ since
they differ in at least $C_i$, and hence everywhere on  $\cR_i\setminus\cR$. Let
$\cW_i$ be the new outside wall of $\cR_i$. See Figure~\ref{apartment
construction}.

\refstepcounter{picture}
\begin{figure}[htbp]\label{apartment construction}
{}\hfill
\beginpicture
\setcoordinatesystem units <0.5cm,0.866cm>    
\setplotarea x from -6 to 6, y from -4 to 4         
\put {$\bullet$}     at 0 0
\put {$\bullet$}     at 1.5 1.5
\put {$\cW$}  [r]         at 2.9  3
\put {$\cW_i$}  [l]         at 6.1  3
\put {$\cR$}           at -6  1.5
\put {$\cR_i\setminus\cR$}    [b]       at -1.2  -3
\put {$v$}     [b r]   at    -0.4   0.1
\put {$E_1$} [b r]  at    0.7   0.8
\put {$C_i$}           at    1.5   0.5
\putrule from -6    0     to  6    0
\arrow <0.2cm> [.2,.67] from 0.6 0.6 to 0.86 0.86
\setlinear
\plot -3  3   3 -3 /
\plot -3 -3   3  3 /
\plot  0 -3   6  3 /
\plot 3 0   1.5  1.5 /
\vshade -3 -3 -3 <z,z,z,z> 0 -3 0 <z,z,z,z> 3 0 3 <z,z,z,z> 6 3 3   /
\endpicture
\hfill{}
\caption{Extending the root $\cR$ to a root $\cR_i$.}
\end{figure}

Now repeat the process for each root $\cR_i$. Continuing in this manner,
we obtain apartments containing the vertex $v\in\D$. The contractability of
$\D$ implies that these apartments intersect only in $\cR$. Using this method we
have constructed at least as many apartments containing $v$ as there are binary
expansions.  We therefore conclude that the number of apartments containing $v$
is uncountable.
\end{remark}

\subsection{The Boundary and its Topology}\label{boundary and topology}

Let $\Om$ be the set of equivalence classes of sectors in $\D$. Given any
$\w\in\Om$ and any fixed vertex $v$ there is a unique sector $\cS_v(\w)$ in
the equivalence class $\w$ based at $v$. The unique chamber $C\in\cS_v(\w)$
containing $v$ is called the {\bf base chamber} of $\cS_v(\w)$. Suppose that
$\cS_v(\w)$ has base chamber $C$. For each vertex $u\in\cS$, we denote the
convex hull of $u$ and $C$ by $I_{v}^u(\w)$. We illustrate such a
convex hull $I_{v}^u(\w)$ as a shaded region in Figure~\ref{initial segment}.

\refstepcounter{picture}
\begin{figure}[htbp]\label{initial segment}
{}\hfill
\beginpicture
\setcoordinatesystem units <0.5cm,0.866cm>    
\setplotarea x from -6 to 6, y from -1 to 6         
\put {$\bullet$} at 0 0
\put {$\bullet$} at 1 5
\put {$v$}                 [t]  at    0  -0.2
\put {$u$}                 [b]  at    1  5.2
\put {$\cS_v(\w)$} [b]  at   -1  5.5
\put {$I_{v}^u(\w)$}      at    0  2
\put {$C$}           at    0  0.6
\putrule from -1    1     to  1    1
\setlinear
\plot -6  6   0 0   6 6 /
\plot -2  2   1 5   3 3 /
\vshade -2 2 2 <z,z,z,z> 0 0 4 <z,z,z,z> 1 1 5 <z,z,z,z> 3 3 3 /
\endpicture
\hfill{}
\caption{Region $I_{v}^u(\w)$ of $\cS_v(\w)$.}
\end{figure}

$\Om$ is in fact the set of chambers of the spherical building $\D^{\infty}$ at
infinity associated to $\D$. We shall refer to $\Om$ as the {\bf boundary}
of $\D$.
This boundary is a totally disconnected compact Hausdorff space with a base for
the topology consisting of sets of the form
\[
U_{v}^u(\w) = \{ \w'\in\Om : I_{v}^u(\w)\subseteq\cS_v(\w') \}
\]
where $v$ is an arbitrary vertex, $\w\in\Om$ and $u$ isany vertex in
$\cS_v(\w)$~\cite{CMS}. Note that
$U_{v}^u(\w' ) = U_{v}^u(\w)$ whenever $\w'\in U_{v}^u(\w)$. The sets
$U_{v}^u(\w)$
form a basis of open and closed sets for the topology on $\Om$. This topology is
independent of the vertex $v$~\cite[Lemma~2.5]{CMS}.

\subsection{Triangle Groups}\label{intro to triangle groups}

Let $(P,L)$ be a finite projective plane of order $q$. Thus there are
$|P|=q^2+q+1$
points and $|L|=q^2+q+1$ lines with each point lying on $q+1$ lines and
each line
having $q+1$ points lying on it. Let $\lambda : P\rightarrow L$ be a point-line
correspondence, i.e.\ a bijection between the elements of $P$ and those of $L$.

Let $\cT$ be a set of triples $\left\{ (x,y,z) : x,y,z\in P\right\}$
satisfying the
following properties:
\begin{enumerate}
\item For all $x,y\in P$, $(x,y,z)\in\cT$ for some $z\in P \Leftrightarrow
y$ and
$\lambda(x)$ are incident.
\item $(x,y,z)\in\cT \Rightarrow (y,z,x)\in\cT$.
\item For all $x,y\in P$, $(x,y,z)\in\cT$ for at most one $z\in P$.
\end{enumerate}

Such a set $\cT$ is called a {\bf triangle presentation}. The notion of triangle
presentations was introduced and developed in~\cite{CMSZ} and the reader is
referred to these papers for proofs of statements regarding triangle
presentations quoted below without reference.

Let $\{ a_x : x\in P\}$ be a set of $q^2+q+1$ distinct letters and define a
multiplicative group $\G$, whose identity shall be denoted by $e$, with the
following presentation:
\[
\G= \left< a_x : a_x a_y a_z =e \text{ for all } (x,y,z)\in\cT\right> .
\]
The Cayley graph of $\G$ constructed via right multiplication with respect
to the
generators $a_x, x\in P$ and their inverses is in fact the skeleton of a
triangle
building $\D_{\cT}$ whose chambers can be identified with the sets $\{ g,
ga_x^{-1},
ga_y\}$ for $g\in\G$ and where $(x,y,z)\in\cT$ for some $z\in P$.  Furthermore,
$\G$ acts simply transitively by left multiplication on the vertices of
$\D_{\cT}$ in
a type-rotating manner. Henceforth we assume that $\D=\D_{\cT}$ for some
triangle presentation $\cT$.

Given a sector $\cS(\w)\subset\D$, left multiplication of every vertex in
$\cS(\w)$ by an element $g\in\G$ defines an action of $\G$ on $\Om$. We
refer the
reader to~\cite{CMS} for a proof that this action is in fact well-defined.

The boundary $\Om$ can be expressed as a union of $(q^2+q+1)(q+1)$ disjoint sets
since, for any fixed vertex $v$,
\[
\Om = \bigcup_{ \text{chambers } C  \text{with } v\in C}\Om_v^C
\]
where $\Om_v^C$ denotes the set of $\w\in\Om$ whose representative sector
based at $v$ has base chamber $C$.

For brevity we will denote the region $I_{e}^u(\w)$ simply by $I^u(\w)$.

\subsection{Labelling Apartments with Elements of $\G$}\label{labelling
apartments}

Suppose we are given a triangle group $\G$ and its corresponding triangle
building $\D$. Each of the edges in $\D$ can be labelled by a generator or
an inverse
generator of $\G$. By providing each edge with an orientation, it suffices
to label
the edges with generators of $\G$. For each $g\in\G$ and each $x\in P$, we
label the
edge $g\rightarrow ga_x$ in $\D$ by $a_x$ or more briefly by $x$ if there is no
likelihood of confusion.

Suppose we are given an apartment $\cA$ and a sector $\cS\subset\cA$ with base
vertex $v$. Each sector wall of $\cS$ is half of a wall in $\cA$.  Each
such wall
consists entirely of edges with the same orientation. Let $\cW_+$ be the wall
containing the sector panel whose orientation emanates from $v$, and denote the
other sector wall of $\cS$ by $\cW_-$. The walls $\cW_+$ and $\cW_-$ and their
translates in $\cA$ form a lattice in $\cA$. Thus each vertex in $\cA$ can
be given a
coordinate with respect to this lattice where the sector panels of $\cS$
are taken
to be in the positive direction and the coordinate in the direction of
$\cW_+$ is
given first. Hence $\cA$ can be determined via its vertices as
$\cA=\left(a_{i,j}\right)_{i,j\in\ZZ}$ where each $a_{i,j}\in\G$. See
Figure~\ref{labelling} for an example.

\refstepcounter{picture}
\begin{figure}[htbp]
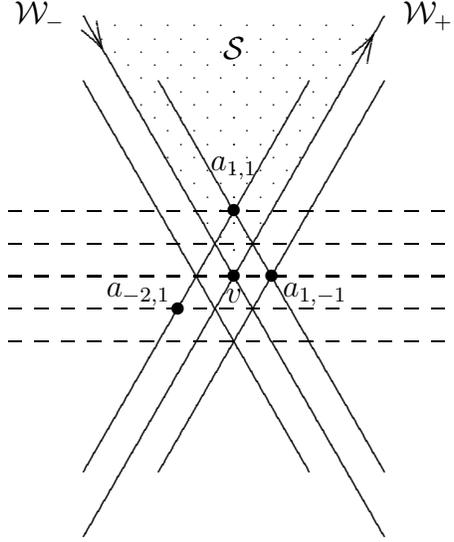
\label{labelling}
{}\hfill
\beginpicture
\setcoordinatesystem units <0.5cm,0.866cm>    
\setplotarea x from -6 to 6, y from -4 to 4         
\put {$\bullet$} at 0 0
\put {$\bullet$} at 0 1
\put {$\bullet$} at 1 0
\put {$\bullet$} at -1.5 -0.5
\put {$v$}             [t]  at    0   -0.2
\put {$a_{1,1}$}   [b]  at    0  1.5
\put {$a_{1,-1}$} [t l]  at    1.3   -0.1
\put {$a_{-2,1}$} [b r]  at    -1.7   -0.45
\put {$\cS$}   at    0 3.5
\put {$\cW_+$}  [l] at    4.5  4
\put {$\cW_-$}  [r] at   -4.5  4
\setlinear
\plot -4  3   2 -3 /
\plot -4  4   4 -4 /
\plot -2  3   4 -3 /
\plot -4 -3   2  3 /
\plot -4 -4   4  4 /
\plot -2 -3   4  3 /
\arrow <10pt> [.3,.67] from 3.5 3.5 to 3.7 3.7
\arrow <10pt> [.3,.67] from -3.7 3.7 to -3.5 3.5
\setdashes
\putrule from -6    0     to  6    0
\putrule from -6    0.5     to  6    0.5
\putrule from -6    1     to  6    1
\putrule from -6   -1     to  6   -1
\putrule from -6   -0.5     to  6   -0.5
\vshade -4 4 4 <z,z,z,z> 0 0 4 <z,z,z,z> 4 4 4 /
\endpicture
\hfill{}
\caption{labelling of an apartment with respect to a sector $\cS$.}
\end{figure}

We note that this induces a labelling of $\cS$ which is independent of the
particular apartment $\cA$ provided that $\cS\subset\cA$.

\subsection{Borel Probability Measures on $\Om$}\label{measure space omega}

We refer the reader to~\cite{CM} and~\cite{CMS} for details of the results
quoted
in this section. For each vertex $v\in\D$, we denote by $V^{m,n}_v$ the set of
vertices $u\in\D$ for which there exists a sector $\cS_v$ based at $v$ such that
$u=a_{m,n}$ with respect to the labelling of $\cS_v$ described
in~\S\ref{labelling
apartments}. Thus $d(v,u)=m+n$ for all $u\in V^{m,n}_v$. The cardinalities
$N_{m,n}=\left| V^{m,n}_v \right|$ are independent of $v$ and are given by
\[
N_{m,n} =
\begin{cases}
(q^2 +q+1)(q^2+q) q^{2(m+n-2)} & \text{if }m,n\geq 1 , \\
(q^2 +q+1)q^{2(m-1)} & \text{if }m\geq 1 , n=0 , \\
(q^2 +q+1)q^{2(n-1)} & \text{if }n\geq 1 , m=0 , \\
1 & \text{if }m=0=n.
\end{cases}
\]

For every pair of vertices $v,u\in\D$ we define
\[
\Om_v^u = \left\{ \w\in\Om : u\in\cS_v(\w)\right\}.
\]
Thus, in terms of the open sets described in~\S\ref{boundary and topology},
\[
\Om_v^u = \bigcup_{\w\in\Om, u\in\cS_v(\w) }U_v^u(\w)
\]
although there is only ever one distinct set $U_v^u(\w)$ contributing to
this union
unless $v$ and $u$ lie on a common wall in $\D$. If $v$ and $u$ lie on a
common wall
in $\D$, there are precisely $q+1$ distinct and disjoint sets of the form
$U_v^u(\w)$ in the above union.

It was noted in~\cite{CMS} that, for each vertex $v\in\D$, there exists a
natural
Borel probability measure $\nu_v$ on $\Om$ which assigns equal measure to each
of the $N_{m,n}$ disjoint sets $\Om_v^u$ for $u\in V^{m,n}_v$. Furthermore,
given
any two vertices $u,v\in\D$, the measures $\nu_u$ and $\nu_v$ are mutually
absolutely continuous~\cite[Lemma~2.5]{CMS}. We define $\nu=\nu_e$ for
brevity and denote $L^\infty(\Om,\nu)$ by $L^\infty(\Om)$.

For any $g\in\G$ and any vertex $v\in\D$, $\nu_{gv}=g\nu_v$ in the sense that
$\nu_{gv}(E)=\nu_v(g^{-1}E)$ for any Borel set $E\subseteq\Om$. Since $\nu$ is
therefore quasi-invariant under the action of $\G$, the von Neumann algebra
$L^\infty(\Om)\rtimes\G$ is well-defined.

\section{Periodicity in Buildings}

\subsection{Periodic Apartments}
Suppose that we have an apartment $\cA$ labelled with respect to some sector
$\cS\subset\cA$ as described in \S\ref{labelling apartments}.

\begin{defns}
We define the {\bf lattice of periodicity points}, $\cL$, of $\cA$ by
\[
\cL=\{ (r,s)\in\ZZ^2 : a_{i,j}^{-1}a_{k,l}=a_{i+r,j+s}^{-1}a_{k+r,l+s} \quad
\text{ for all }  i,j,k,l\in\ZZ \}.
\]
To appreciate the meaning of $\cL$, note that there is a natural action by
translation of $\ZZ^2$ on the apartment $\cA$ given by
$(r,s)a_{i,j}=a_{i+r,j+s}$.
Suppose $a_{i,j}$ and $a_{k,l}$ are adjacent vertices. The element
$a_{i,j}^{-1}a_{k,l}$ defines a labelled edge from $a_{i,j}$ to $a_{k,l}$. If
$(r,s)\in\cL$ then the element $a_{i+r,j+s}^{-1}a_{k+r,l+s}$ defines exactly the
same labelling on the edge from $a_{i+r,j+s}$ to $a_{k+r,l+s}$. Thus a translation by
$(r,s)\in\cL$ leaves the edge labelling of the apartment invariant.

Note that $\cL$ is a subgroup of $\ZZ^2$, so that we must have either $\cL=\{
(0,0)\}$, $\cL\cong \ZZ$ or $\cL\cong\ZZ^2$. We call the elements
$(r,s)\in\cL$ the
{\bf periodicity points} of $\cA$.

We say that $\cA$ is {\bf periodic} if $\cL$ is non-trivial. We distinguish
between
the cases $\cL\cong \ZZ$ and $\cL\cong\ZZ^2$ by saying that $\cA$ is {\bf singly
periodic} or {\bf doubly periodic} respectively. We say that $\cA$ is
{\bf $(r,s)$-periodic} if $(r,s)\in\cL$ regardless of whether $\cA$ is
singly or doubly
periodic. In geometric terms this means that the labelling of the directed edges
of $\cA$ by generators of $\G$ has a translational symmetry in the $(r,s)$
direction. For brevity we say $\cA$ is {\bf $\left\{
(r,s),(t,u)\right\}$-periodic} if
$\cA$ is both $(r,s)$-periodic and $(t,u)$-periodic.

We note that the set of periodic apartments in $\D$ is $\G$-invariant since if
$\cA$ is periodic then $g\cA$ is necessarily periodic and has the same lattice
of periodicity points for all $g\in\G$.

Given a periodic apartment $\cA$ with periodicity lattice $\cL$, define
$m\in\ZZ$
by
\[
m =\min_{r,s\in\cL, (r,s)\neq (0,0)}
\left| a_{i,j}^{-1} a_{i+r,j+s}  \right|
\]
for any $i,j\in\ZZ$. Such an integer exists and is positive since the
length of a word
in
$\G$ is a positive integer and $\cL$ is non-trivial since we are assuming
periodicity. We say that $\cA$ has {\bf minimal period} $m$.
\end{defns}

\subsection{Rigidly Periodic Apartments}
We introduce some periodic apartments whose behaviour is somewhat special.
\begin{defn} Call an apartment $\cA$ {\bf rigidly periodic} if
either
\begin{itemize}
\item $\cA$ is doubly periodic, or
\item $\cA$ is $(r,s)$-periodic with $r,s\neq 0$ and $s \neq -r$.
\end{itemize}
\end{defn}
Thus singly periodic apartments whose vertices $a_{r,s}, (r,s)\in\cL$
did not lie entirely along a single wall in the apartment would be rigidly
periodic.
In fact, it turns out that rigid and double periodicity are equivalent
notions,  see Lemma~\ref{rigid implies doubly periodic}.

Note that if $\cA$ is rigidly periodic then so is $g\cA$ for all $g\in\G$
so that the
set of rigidly periodic apartments is $\G$-invariant. Since $\G$ is
countable, there
can only be a countable number of rigidly periodic apartments containing a fixed
vertex $v\in\D$. Recall that there are an uncountable number of apartments
containing $v$ by Remark~\ref{uncountable apartments}.  We deduce that, for
any vertex $v\in\D$, there exist apartments containing $v$ which are not rigidly
periodic.

The existence of rigidly periodic apartments with arbitrarily large minimal
periodicity was established initially by S.\ Mozes~\cite[Theorem~2.2']{Moz} who
showed that periodic apartments are dense in the case where $\G$ embeds as a
cocompact lattice in a strongly transitive group of automorphisms of $\D$. The
existence of rigidly periodic apartments in the context of $\tilde{A}_2$
groups was  established by A.\ M.\ Mantero, T.\ Steger, and A.\ Zappa
in~\cite[Lemma 2.4 and Proposition 2.5]{MSZ}. We outline their proof since it is
constructive and therefore more useful than the mere existence result and we
refer the reader to~\cite{MSZ} for the full details.

\begin{thm}\label{existence}
There exist rigidly periodic apartments with arbitrarily large minimal period.
\end{thm}
\begin{proof}
Suppose we wish to construct a rigidly periodic apartment with minimal period
greater than a fixed positive integer $m$.  We begin by fixing a non-periodic
apartment $\cA$ containing $e$ which is labelled with respect to some sector
$\cS$ based at $e$. Thus $\cA=(a_{i,j})$ with $a_{0,0} =e$. Denote by $C$
the base
chamber of $\cS$. Let $v=a_{m,m}\in\cA$ and denote by $R$ the convex hull of
$C$ and $v$ in $\cA$. $R$ is a diamond with the vertices $e$ and $v$ at opposite
corners. Let $C'$ be the unique chamber in $R$ containing $v$.

Suppose that the edge on $C$ opposite $e$ is labelled $a$, andthat the edge on
$C'$ opposite $v$ is labelled $a'$, see Figure~\ref{convex hull of e and v}.

\refstepcounter{picture}
\begin{figure}[htbp]
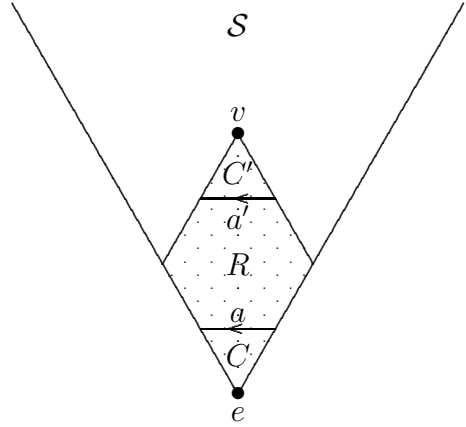
\label{convex hull of e and v}
{}\hfill
\beginpicture
\setcoordinatesystem units <0.5cm,0.866cm>    
\setplotarea x from -6 to 6, y from -1 to 6         
\put {$\bullet$} at 0 0
\put {$\bullet$} at 0 4
\put {$e$}     [t]  at    0  -0.2
\put {$v$}     [b] at    0  4.2
\put {$\cS$} [b]  at    0  5.5
\put {$R$}           at    0  2
\put {$C$}           at    0  0.6
\put {$C'$}          at    0  3.4
\put {$a$}    [b]  at    0  1.1
\put {$a'$}   [t]   at    0  2.9
\putrule from -1    1     to  1    1
\putrule from -1    3     to  1    3
\setlinear
\plot -6  6   0 0   6 6 /
\plot -2  2   0 4   2 2 /
\vshade -2 2 2 <z,z,z,z> 0 0 4 <z,z,z,z>  2 2 2 /
\arrow <5pt> [.3,.67] from 0.1 1 to -0.25 1
\arrow <5pt> [.3,.67] from 0.25 3 to -0.1 3
\endpicture
\hfill{}
\caption{Convex hull $R$ of $C$ and $v$.}
\end{figure}

By altering our choice of non-periodic apartment $\cA$ if necessary, we may
assume that $a\neq a'$ and that $R$ is not contained in a rigidly periodic
apartment with minimal period less than $m$.

Results in~\cite{MSZ} prove that there exist chambers $D$ and $D'$ satisfying
the following conditions:
\begin{enumerate}
\item  $D$ is opposite $C$ in the residue determined by~$e$.
\item $D'$ is opposite $C'$ in the residue determined by~$v$.
\item $D$ corresponds to a relation $bcd=e$ and the edge labelled $b$ is
opposite
the vertex~$e$.
\item $D'$ corresponds to a relation $bc'd'=e$ and the edge labelled $b$ is
opposite
the vertex~$v$.
\item $c\neq c'$ and $d\neq d'$.
\end{enumerate}
Thus we can extend $R$ non-trivially at both ends by $D$ and $D'$ to construct a
region $R'$, pictured in Figure~\ref{region R'},
\refstepcounter{picture}
\begin{figure}[htbp]
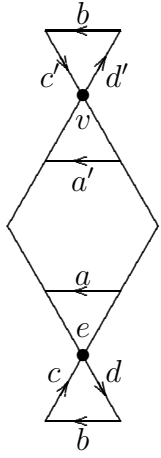
\label{region R'}
{}\hfill
\beginpicture
\setcoordinatesystem units <0.5cm,0.866cm>    
\setplotarea x from -6 to 6, y from -2 to 6         
\put {$\bullet$} at 0 4
\put {$\bullet$} at 0 0
\put {$e$}     [b]  at    0  0.3
\put {$v$}     [t] at    0  3.7
\put {$a$}    [b]  at    0  1.1
\put {$b$}    [b]  at    0  5.1
\put {$b$}    [t]  at    0  -1.1
\put {$a'$}   [t]   at    0  2.9
\put {$c$}    [br]  at    -0.6  -0.4
\put {$c'$}    [tr]  at    -0.6  4.5
\put {$d$}    [bl]  at    0.6  -0.4
\put {$d'$}    [tl]  at    0.6  4.5
\putrule from -1    1     to  1    1
\putrule from -1    5     to  1    5
\putrule from -1    3     to  1    3
\putrule from -1    -1     to  1    -1
\setlinear
\plot 1 5   -2 2   1 -1  /
\plot -1 5   2 2   -1 -1   /
\arrow <5pt> [.3,.67] from 0.1 1 to -0.25 1
\arrow <5pt> [.3,.67] from 0.1 3 to -0.25 3
\arrow <5pt> [.3,.67] from 0.1 5 to -0.25 5
\arrow <5pt> [.3,.67] from 0.1 -1 to -0.25 -1
\arrow <5pt> [.3,.67] from -0.6 4.6 to -0.4 4.4
\arrow <5pt> [.3,.67] from 0.4 4.4 to 0.6 4.6
\arrow <5pt> [.3,.67] from -0.6 -0.6 to -0.4 -0.4
\arrow <5pt> [.3,.67] from 0.4 -0.4 to 0.6 -0.6
\endpicture
\hfill{}
\caption{The region $R'$.}
\end{figure}
which undergoes no cancellation upon iteration. This enables us to construct an
apartment $\cA'$ containing the region obtained by tesselating $R'$
infinitely in
the directions of $\cS$ and its opposite as in Figure~\ref{tesselation}.

\refstepcounter{picture}
\begin{figure}[htbp]
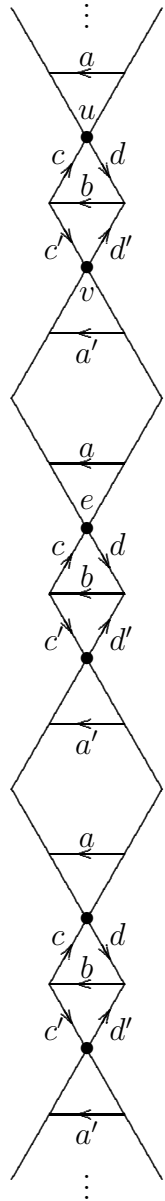
\label{tesselation}
{}\hfill
\beginpicture
\setcoordinatesystem units <0.5cm,0.866cm>    
\setplotarea x from -6 to 6, y from -11 to 10         
\multiput {$\bullet$} at 0 6 *2 0 -6 /
\multiput {$\bullet$} at 0 4 *2 0 -6 /
\put {$\vdots$}   at    0  8
\put {$\vdots$}   at    0  -10
\put {$e$}     [b]  at    0  0.3
\put {$v$}     [t] at    0  3.7
\put {$u$}     [b] at    0  6.3
\multiput {$a$}    [b]  at    0  7.1 *2 0 -6 /
\multiput {$b$}    [b]  at    0  5.1 *2 0 -6 /
\multiput {$a'$}   [t]   at    0  2.9 *2 0 -6 /
\multiput {$c$}    [br]  at    -0.6  5.6 *2 0 -6 /
\multiput {$c'$}    [tr]  at    -0.6  4.5 *2 0 -6 /
\multiput {$d$}    [bl]  at    0.6  5.6 *2 0 -6 /
\multiput {$d'$}    [tl]  at    0.6  4.5 *2 0 -6 /
\putrule from -1    1     to  1    1
\putrule from -1    5     to  1    5
\putrule from -1    -5     to  1    -5
\putrule from -1    3     to  1    3
\putrule from -1    -3     to  1    -3
\putrule from -1    7     to  1    7
\putrule from -1    -1     to  1    -1
\putrule from -1    -7     to  1    -7
\putrule from -1    -9     to  1    -9
\setlinear
\plot -2  8   1 5   -2 2   1 -1  -2 -4   1 -7   -2 -10 /
\plot 2  8   -1 5   2 2   -1 -1  2 -4   -1 -7   2 -10 /
\arrow <5pt> [.3,.67] from 0.1 1 to -0.25 1
\arrow <5pt> [.3,.67] from 0.1 3 to -0.25 3
\arrow <5pt> [.3,.67] from 0.1 5 to -0.25 5
\arrow <5pt> [.3,.67] from 0.1 7 to -0.25 7
\arrow <5pt> [.3,.67] from 0.1 -1 to -0.25 -1
\arrow <5pt> [.3,.67] from 0.1 -3 to -0.25 -3
\arrow <5pt> [.3,.67] from 0.1 -5 to -0.25 -5
\arrow <5pt> [.3,.67] from 0.1 -7 to -0.25 -7
\arrow <5pt> [.3,.67] from 0.1 -9 to -0.25 -9
\arrow <5pt> [.3,.67] from -0.6 5.4 to -0.4 5.6
\arrow <5pt> [.3,.67] from 0.4 5.6 to 0.6 5.4
\arrow <5pt> [.3,.67] from -0.6 4.6 to -0.4 4.4
\arrow <5pt> [.3,.67] from 0.4 4.4 to 0.6 4.6
\arrow <5pt> [.3,.67] from -0.6 -0.6 to -0.4 -0.4
\arrow <5pt> [.3,.67] from 0.4 -0.4 to 0.6 -0.6
\arrow <5pt> [.3,.67] from -0.6 -1.4 to -0.4 -1.6
\arrow <5pt> [.3,.67] from 0.4 -1.6 to 0.6 -1.4
\arrow <5pt> [.3,.67] from -0.6 -6.6 to -0.4 -6.4
\arrow <5pt> [.3,.67] from 0.4 -6.4 to 0.6 -6.6
\arrow <5pt> [.3,.67] from -0.6 -7.4 to -0.4 -7.6
\arrow <5pt> [.3,.67] from 0.4 -7.6 to 0.6 -7.4
\endpicture
\hfill{}
\caption{Infinite tesselation of the region $R'$.}
\end{figure}

The convex hull of the infinite tesselation of $R'$ uniquely determines the
entire apartment $\cA'$. Let
 \[
u=a_{m+2,m+2}=vd'd^{-1}=v(c')^{-1}c
\]
as labelled on Figure~\ref{tesselation}. Note that the labelling of the edges of
$\cA'$ looks the same when viewed from each of the vertices $u^n\in\cA'$  and
hence $u^n\cA'=\cA'$ for every $n\in\ZZ$. Thus $\cA'$ is $(m+2,m+2)$-periodic
with minimal periodicity greater than or equal to $m$ by construction and its
periodicity lattice will contain the elements $\left(n(m+2),n(m+2)\right)$
for all
$n\in\ZZ$. \end{proof}

The following observation enables us to deduce some rather nice
properties of rigidly periodic apartments and the boundary points they define.

\begin{lem}\label{rigidity}
A rigidly periodic apartment is completely determined by any single sector.
\end{lem}
\begin{proof}
We begin by labelling the apartment $\cA$  with respect to the known sector
$\cS$ so that
the set of vertices  $\{ a_{k,l} : k,l\geq 0 \}$ is known. We show that in
fact $a_{i,j}$ is then uniquely determined for all $i,j\in\ZZ$. Without loss of
generality, we consider three separate cases.
\par\noindent{\em Case 1: The doubly periodic case.}

Suppose that $\cA$  is an $\left\{ (r,s),(t,u)\right\}$-periodic apartment with
$(r,s)$ and $(t,u)$ being linearly independent vectors in $\RR^2$. Thus,
given any
$i,j\in\ZZ$ there exist $m,n\in\ZZ$ such that $mr+nt,ms+nu,i+mr+nt$, and
$j+ms+nu$
are all positive. From the properties of double periodicity, we deduce that
\[
a_{i,j}=a_{k,l}a_{k+mr+nt,l+ms+nu}^{-1}a_{i+mr+nt,j+ms+nu}
\]
so that if $k,l\geq 0$ every term on the right hand side of the equation is
a vertex
in
$\cS$ and is therefore known. Since $i$ and $j$ were arbitrary integers,
$a_{i,j}$ is
uniquely determined for all $i,j\in\ZZ$. We deduce that $\cA$ is completely
determined.
\par\noindent{\em Case 2: The $(r,s)$-periodic case with $r,s>0$.}

In this case given any $i,j\in\ZZ$ there exists a positive integer $m$ such that
$i+mr,j+ms>0$. Once again this leads to an equation of the form
\[
a_{i,j}=a_{k,l}a_{k+mr,l+ms}^{-1}a_{i+mr,j+ms}
\]
where every element on the right hand side is a vertex in $\cS$ and therefore
uniquely determined.
\par\noindent{\em Case 3: The $(r,s)$-periodic case with either $r>0$ but
$s<0$ or
$r<0$ but $s>0$.}

We note that the vertices $a_{mr,ms}$ are determined for all $m\in\ZZ$ since
\[
a_{mr,ms}= a_{k+mr,l+ms}a_{k,l}^{-1}a_{0,0}
\]
and suitably large choices of $k,l$ ensures that all the elements on the
right hand
side are vertices in $\cS$.Since  $r\neq 0\neq s\neq -r$ the vertex $a_{mr,ms}$
does not lie on a sector wall emanating from $a_{0,0}$. Suppose that $\cS$
has base
chamber $C$. The convex hull of $C$ and $a_{mr,ms}$ for all $m\in\ZZ$ must
also be
contained in $\cA$ . By taking arbitrarily large positive values for $m$ we know
that a sector $\cS'$ adjacent to $\cS$ is determined. Taking arbitrarily large
negative values for $m$ implies that  the sector $\cS''$ opposite $\cS'$ is
determined. Since the convex hull of $\cS'$ and $\cS''$ is $\cA$  this
means that
$\cA$  is entirely determined. Figure~\ref{case 3} illustrates the proof in
the case
of $r>0$ and $s<0$ where the shaded area represents the convex hull of $C$ and
$a_{mr,ms}$.

\refstepcounter{picture}
\begin{figure}[htbp]
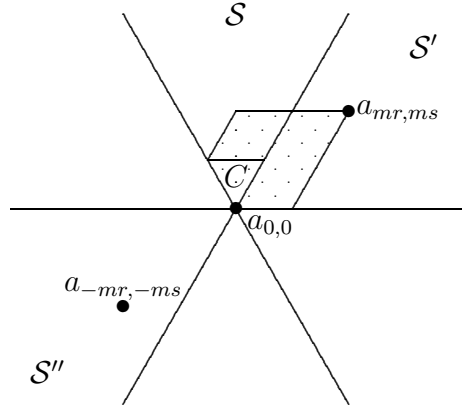
\label{case 3}
{}\hfill
\beginpicture
\setcoordinatesystem units <0.5cm,0.866cm>    
\setplotarea x from -6 to 6, y from -4 to 4         
\put {$\bullet$} at 0 0
\put {$\bullet$} at -3 -1.5
\put {$\bullet$} at 3 1.5
\put {$a_{0,0}$}             [t l]  at    0.3   -0.1
\put {$a_{mr,ms}$}       [l] at   3.2  1.5
\put {$a_{-mr,-ms}$}   [b] at   -3  -1.4
\put {$C$}   at   0  0.5
\put {$\cS$}    at   0  3
\put {$\cS'$}   at    5 2.5
\put {$\cS''$}  at    -5  -2.5
\putrule from -6    0     to  6    0
\putrule from -0.75 0.75  to  0.75 0.75
\putrule from   0    1.5     to  3 1.5
\setlinear
\plot -0.75 0.75   0 1.5 /
\plot  1.5 0   3 1.5 /
\plot -3  3   3 -3 /
\plot -3 -3   3  3 /
\vshade -0.75 0.75 0.75 <z,z,z,z>    0 0 1.5  <z,z,z,z> 1.5 0 1.5
<z,z,z,z>    3 1.5 1.5 /
\endpicture
\hfill{}
\caption{Illustration of the proof of case 3 with $r>0$ and $s<0$.}
\end{figure}

This completes the proof since all other cases are equivalent to one of those
considered.
\end{proof}

\begin{cor}\label{convex hull enough}
Suppose an apartment $\cA=(a_{i,j})_{i,j\in\ZZ}$ is rigidly periodic. If
$\cA$ is
$(r,s)$-periodic for some $r,s\in\ZZ$ with $r,s\neq 0$ and $s\neq -r$ then
$\cA$ is
uniquely determined by $a_{i,j}$, $a_{i+r,j+s}$ and
$a_{i,j}a_{i+r,j+s}^{-1}$ for every
$i,j\in\ZZ$.
\end{cor}
\begin{proof}
Let $g=a_{i,j}^{-1}a_{i+r,j+s}$ and $\cC$ be the convex hull of $a_{i,j}$ and
$a_{i+r,j+s}$ as in Figure~\ref{convex hull in periodic apartment}.

\refstepcounter{picture}
\begin{figure}[htbp]
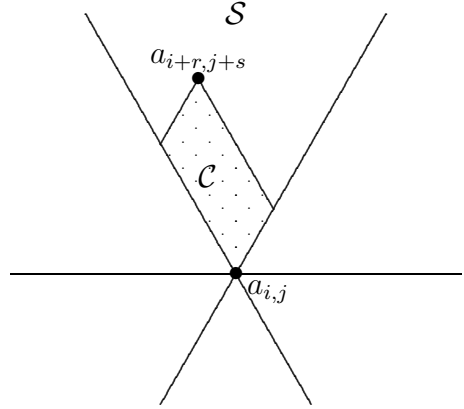
\label{convex hull in periodic apartment}
{}\hfill
\beginpicture
\setcoordinatesystem units <0.5cm,0.866cm>    
\setplotarea x from -7 to 7, y from -2 to 5         
\put {$\bullet$} at 0 0
\put {$\bullet$} at -1 3
\put {$a_{i,j}$}             [t l]  at    0.3   -0.1
\put {$a_{i+r,j+s}$}             [b] at   -1 3.1
\put {$\cC$}                         at   -0.75 1.5
\put {$\cS$}                         at   0 4
\putrule from -6    0     to  6    0
\setlinear
\plot -4  4   2 -2 /
\plot -2 -2   4  4 /
\plot  -2 2   -1 3    0 2   1 1 /
\vshade -2 2 2 <z,z,z,z>    -1 1 3  <z,z,z,z> 0 0 2  <z,z,z,z>    1 1 1 /
\endpicture
\hfill{}
\caption{Convex hull $\cC$ of $a_{i,j}$ and $a_{i+r,j+s}$.}
\end{figure}

Then $g^n\cA=\cA$ for every $n\in\ZZ$ and the convex hull of the iterates
$g^n\cC$, for $n\in\NN$ determine a sector $\cS$ in $\cA$. Thus $\cA$ is
uniquely
determined by Lemma~\ref{rigidity}.
\end{proof}

\begin{cor}\label{cor 1}
A boundary point $\w\in\Om$ can have representative sectors in at most one
rigidly periodic apartment $\cA$ .
\end{cor}
\begin{proof}
Suppose $\cA$  and $\cA'$  are both rigidly periodic apartments which contain
representative sectors of $\w$; say $\cS\subset \cA$ and $\cS'\subset \cA'$ are
two such sectors. Since $\cS$ and $\cS'$ are equivalent sectors, they must
contain a
common subsector $\cS''$. Hence $\cS''\subseteq \cA\cap \cA'$, and therefore by
Lemma~\ref{rigidity} we must have $\cA=\cA'$.
\end{proof}

The following corollary is a generalization of~\cite[Lemma 3.4]{MSZ}.

\begin{cor}\label{cor 2}
Let $\cA$  and $\cA'$  be rigidly periodic apartments which contain
representative
sectors of $\w,\w'\in\Om$ respectively. If $g\w=\w'$ for some $g\in\G$ then
$\cA'=g\cA$. In particular, if $\cA=\cA'$ then $g$ must stabilize $\cA$.
\end{cor}
\begin {proof}
Suppose $\cS_v(\w)\subset \cA$ is a representative of $\w$. The apartment
$g\cA$ is
rigidly periodic and contains $g\cS_v(\w)=\cS_{gv}(g\w)=\cS_{gv}(\w')$, a
representative of $\w'$. By Corollary~\ref{cor 1} we must have $\cA'=g\cA$.
\end{proof}

Notice that it is not necessary to know {\em a priori} that $\cA$  and
$\cA'$ have
the same minimal period; we simply need to know that they are both rigidly
periodic.

If we know that an apartment is periodic and that it is stabilized by a certain
subset of $\G$, the following result enables us to bound the minimal
periodicity of
$\cA$.

\begin{prop}
Let $F\subset\Gamma\setminus\{ e\}$ be a finite set, and suppose $\cA$ is a
periodic apartment such that $g\cA=\cA$ for all $g\in F$. Then $\cA$ can
have minimal period of at most $2\max\{ |g| : g\in F\}$.
\end{prop}
\begin{proof}
Given such a periodic apartment $\cA$, label its vertices with respect
to some sector $\cS\in\cA$ and suppose that $\cA$ has minimal period $m$. Denote
by $S$ the group of symmetries generated by reflections in $\cA$ fixing
$a_{0,0}$. We begin by noting that if $g\cA=\cA$ then $ga_{i,j}=a_{k,l}$ where
\[
(k,l)= T_g(i,j) = \s(i,j) + (r,s)
\]
for some $\s\in S$ and $r,s\in\ZZ$. This is a slight generalization of a
statement
made in~\cite{MSZ} and we refer the reader to that paper for a proof since it
generalizes in a very straightforward manner.

Since $T_g$ corresponds to a non-trivial symmetry of $\cA$, $(r,s)$ must be
non-trivial whenever $\s$ is; otherwise the symmetry of $\cA$ would be a
reflection and would collapse $\cA$ to a root. Thus each $g\in F$ which
stabilizes
$\cA$ must either induce a translation or a glide reflection on $\cA$. If
$g$ induces
a translation on $\cA$ then $|g|\geq m$. On the other hand, if $g$ induces a
glide-reflection on $\cA$ then $g^2$ must induce a translation on $\cA$ so that
$|g^2|\geq m$. Since  $2|g|\geq |g^2|$ the result follows.
\end{proof}

\begin{cor}\label{m(F)}
Given a finite set $F\subset\Gamma\setminus\{ e\}$, there exists a positive
integer
$m(F)$ such that no element of $F$ can stabilize a periodic apartment with
minimal period greater than $m(F)$.
\end{cor}

If $F=\{ g\}$ we denote $m(F)$ by $m(g)$.

The following result is due to S.~Mozes (c.f.\ proof of Proposition~2.13
of~\cite{Moz}) and was communicated to us by T.~Steger. We include it only
for the
benefit of the curious since we will not need to apply the result.

\begin{lem}\label{rigid implies doubly periodic}
Every rigidly periodic apartment is doubly periodic.
\end{lem}
\begin{proof}
We need to show that every apartment $\cA$ which is $(r,s)$-periodic with
$r,s\neq 0$ and $s\neq-r$ is in fact doubly periodic. We begin by noting
that, with a
suitable relabelling of $\cA$ we may assume $r,s >0$.

By Corollary~\ref{convex hull enough}, the apartment $\cA$ and its labelling is
completely determined by its periodicity and by the convex hull, $\cC$, of
$a_{0,0}$ and
$a_{r,s}$. Consider the strip of the apartment between the wall $\cW$
containing the vertices
$a_{m,-m}$, and the parallel wall $\cW'$ containing the vertices
$a_{r+m,s-m}$. This
strip is depicted in Figure~\ref{strip}, where the shaded area represents
$\cC$. We
note that $\cC$ contains only a finite number of chambers.

\refstepcounter{picture}
\begin{figure}[htbp]
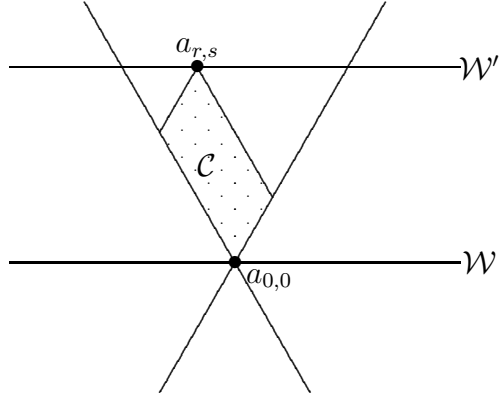
\label{strip}
{}\hfill
\beginpicture
\setcoordinatesystem units <0.5cm,0.866cm>    
\setplotarea x from -7 to 7, y from -2 to 5         
\put {$\bullet$} at 0 0
\put {$\bullet$} at -1 3
\put {$a_{0,0}$}             [t l]  at    0.3   -0.1
\put {$a_{r,s}$}             [b] at   -1 3.1
\put {$\cC$}                         at   -0.75 1.5
\put {$\cW$}                        at    6.5  0
\put {$\cW'$}                       at    6.5  3
\putrule from -6    0     to  6    0
\putrule from -6    3     to  6    3
\setlinear
\plot -4  4   2 -2 /
\plot -2 -2   4  4 /
\plot  -2 2   -1 3    0 2   1 1 /
\vshade -2 2 2 <z,z,z,z>    -1 1 3  <z,z,z,z> 0 0 2  <z,z,z,z>    1 1 1 /
\endpicture
\hfill{}
\caption{Strip between walls $\cW$ and $\cW'$ showing convex hull $\cC$ of
$a_{0,0}$ and $a_{r,s}$.}
\end{figure}
The convex hull of the vertices $a_{m,-m}$ and $a_{m+r,-m+s}$ will have the same
shape as $\cC$ for every $m\in\ZZ$. We consider the labelling of these convex
hulls.  There are only a finite number of possible labellings for a chamber
and hence there are only a finite number of ways these convex hulls can be
labelled. However, the strip in question is of infinite length so there
must be at
least two such convex hulls which share the same labelling. Without loss of
generality we may assume that one of these is $\cC$. Suppose that the other
is the
convex hull $\cC'$ of $a_{n,-n}$ and $a_{n+r,-n+s}$ for some
$n\in\ZZ$.

The $(r,s)$-periodicity and $\cC'$  completely determine $\cA$, as was the case
for $\cC$. Moreover the labelling of $\cA$ induced by this construction is
identical to that produced by $\cC$, but shifted by $(n,-n)$. Hence $\cA$
must be
$(n,-n)$-periodic as well as rigidly periodic. Thus $\cA$ is doubly periodic.
\end{proof}

Lemma~\ref{rigid implies doubly periodic} therefore establishes the equivalence
of rigid periodicity and double periodicity.

\subsection{Periodic Walls, Sectors and Roots}\label{periodic W,S and R}

We note that the notions of periodic apartments and minimal periodicity can be
generalized so that they apply to walls, sectors and roots in $\D$ by imposing
various conditions on the integers $i,j,k,l,i+r,j+s,k+r,l+s$ in the equation
\[
a_{i,j}^{-1}a_{k,l} = a_{i+r,j+s}^{-1}a_{k+r,l+s}  .
\]
The set of pairs $(r,s)\in\ZZ^2$ satisfying such an equation is then no longer
necessarily a subgroup of $\ZZ^2$.

Without loss of generality, we may assume that the vertex $a_{0,0}$ is on
the wall,
or is the base vertex of the sector, or is on the boundary of the root in
question.
We generalize the notions of rigid periodicity and single periodicity to the
context of sectors and roots.

It is not in general true that a periodic root extends to a periodic
apartment or
that such an apartment is unique if it exists. However, the arguments used
in the
proof of Lemma~\ref{rigidity} show that a rigidly periodic sector determines a
unique rigidly periodic apartment. We complete the analysis of periodic sectors
with the following result.

\begin{lem}\label{complete description of periodic sectors}
A periodic sector which is not rigidly periodic determines a unique
periodic root.
\end{lem}
\begin{proof}
Suppose that $\cS$ is a periodic sector whose vertices are labelled with respect
to its base vertex, $a_{0,0}$ in such a way that $a_{i,j}\in\cS$ for all
$i,j\in\NN$.
Suppose also that $\cS$ is not rigidly periodic.

If $\cS$ is $(r,0)$-periodic or $(0,s)$-periodic for some $r$ or $s\in\ZZ$ let
$g=a_{r,0}a_{0,0}^{-1}$ or $g=a_{0,s}a_{0,0}^{-1}$ respectively. If $\cS$ is
$(r,-r)$-periodic for some $r\in\NN$, let $g=a_{r,-r}a_{0,0}^{-1}$.

The sectors $g^n\cS$ for $n\in\ZZ$ then lie in a common apartment since the
labelling of the directed edges in $\cS$ by generators of $\G$ have a
symmetry in
the required direction. There may be several apartments containing the sectors
$g^n\cS$ for $n\in\ZZ$. However the convex hull of the sectors $g^n\cS$ for
$n\in\ZZ$ in any apartment containing them will determine a unique periodic root
$\cR$. In particular, the wall bounding $\cR$ and its translates in $\cR$
are periodic.
\end{proof}

\section{Periodic Limit Points and Their Properties}

\subsection{Periodic Limit Points and Their Behaviour under the Action of $\G$}
We begin by using the notion of periodicity in $\D$ to define a useful subset of
$\Om$.
\begin{defns}
Call a boundary point $\w\in\Om$ a {\bf periodic limit point} if it has a
periodic
representative sector $\cS(\w)$, say. We refer to such a sector $\cS(\w)$ as a
{\bf periodic representative} of $\w$. Note that we do not assume that every
periodic limit point has a periodic representative $\cS(\w)$ with $e\in\cS(\w)$.
We denote the set of periodic limit points by {\boldmath{$\Pi$}}.
\end{defns}

A most intriguing property of periodic limit points, proved forthwith, is
that they
are the only boundary points which can be stable under the action of an element of
$\G$.

\begin{prop}\label{only boundary points are stable}
Suppose $\w\in\Om$ satifies $g\w=\w$ for some $g\in\G\setminus\{ e\}$. Then we
must have~$\w\in\Pi$.
\end{prop}
\begin{proof}
Suppose $\w\in\Om$ satisfies $g\w=\w$ for some $g\in\G$. Then $g^n\w=\w$ for
all $n\in\ZZ$. Let $\cS=\cS(\w)$ be any representative sector of $\w$. Thus
$g^n\cS$ is a representative of $\w$ for every $n\in\ZZ$. Consider first the
representatives $\cS$ and $g\cS$ of $\w$. Note that all the vertices in
$g\cS$ are
of the form $gv$ where $v\in\G$ is a vertex in $\cS$.

Since $\cS$ and $g\cS$ are parallel they must contain a common subsector.
Suppose $\cS_{gv}=\cS_{gv}(\w)\subseteq \cS\cap g\cS$ is such a common
subsector. Since $\cS_{gv}(\w)\subseteq g\cS$, we must have
$g^{-1}\cS_{gv}(\w)\subseteq \cS$, which is to say $\cS_{v}(\w)\subseteq \cS$.
Thus the sectors $\cS_{v}=\cS_{v}(\w)$ and $g\cS_v=\cS_{gv}$ are both contained
in $\cS$ and hence are contained in a common apartment. Since $\cS_{v}$ and
$g\cS_v$ are parallel they must intersect in a sector which is a translate
of each
of them. Thus the labelling of $\cS_v$ must have a $g$-translational symmetry.
The sector $\cS_v$ is therefore a periodic representative of $\w$ and hence
$\w\in\Pi$, thus establishing the claim.
\end{proof}

\begin{defns}
If $\w\in\Om$ has a periodic representative $\cS(\w)$ which is in fact rigidly
periodic, Lemma~\ref{rigidity} and Corollary~\ref{cor 1} enable us to
associate to
$\w$ a unique rigidly periodic apartment $\cA(\w)$. In this case we say that
$\w$ is a {\bf rigidly periodic limit point} and we denote by
{\boldmath{$\La$}} the
subset of $\Om$ consisting of all rigidly periodic limit points.
\end{defns}

\begin{remark}
Given a particular $g\in\G$ there are at most a countable number of rigidly
periodic apartments containing $g$. Since $\G$ is countable this means that
there
are at most a countable number of distinct rigidly periodic apartments in
$\D$ and
therefore $\La$ is a countable subset of $\Om$. Note that $\Om$ is
uncountable by
an argument analogous to that of Remark~{uncountable apartments}. There must
therefore exist boundary points $\w\in\Om$ which are not rigidly periodic limit
points.
\end{remark}

\subsection{Topological Considerations}

Despite the fact that $\Om\setminus\La\neq\emptyset$, we prove that $\La$, and
hence $\Pi$, is a dense subset of $\Om$.

\begin{prop}\label{periodic limit points everywhere}
Every open set $U_{v}^u(\w)\subseteq\Om$ contains rigidly periodic limit points
whose periodic representatives have arbitrarily large minimal period. Moreover
rigidly periodic limit points can be found whose periodic representatives are
$(r,s)$-periodic with
$r,s>0$.
\end{prop}
\begin{proof}
Given any boundary point $\w\in\Om$, any vertex  $x\in\cS_v(\w)$, and any
$m\in\NN$, we can use the proof of Theorem~\ref{existence} to construct a
rigidly
periodic apartment $\cA$ containing $v^{-1}I_{v}^u(\w)$ and whose minimal
periodicity is greater than $m$. The apartment $\cA'=v\cA$ is then a rigidly
periodic apartment and the sector $\cS_v\subset\cA'$ containing
$I_{v}^u(\w)$ defines a rigidly periodic limit point $\w'$ satisfying the
required
conditions.
\end{proof}

\begin{cor}
$\La$ and $\Pi$ are dense $\G$-invariant subsets of $\Om$.
\end{cor}
\begin{proof}
The density follows immediately from Proposition~\ref{periodic limit points
everywhere}. To see that $\La$ and $\Pi$ are both $\G$-invariant, simply
note that
if $\w\in\La$ has a periodic representative $\cS\subset\cA$ then $g\cS\subset
g\cA$ is a periodic representative of $g\w$ of the same type.
\end{proof}

\subsection{Measure-Theoretic Considerations}

In Lemma~\ref{Pi has measure zero} we show that $\nu_v(\Pi)=0$ for every
vertex $v\in\D$. Before we launch into the proof of this result we establish a
necessary technical result.\footnote{The proof of this result is not complete. It is corrected in:
P. Cutting and G. Robertson, Type III actions on boundaries of $\tilde A_n$ buildings, J. Operator Theory, 49 (2003), 25-44.}

\begin{lem}\label{measure of apartments with a periodic wall}
Let $\cW$ be any wall in $\D$. Let $\Si_{\cW}\subseteq\Om$ be the set of
boundary points which have representative sectors in any apartment containing
$\cW$. Then
\[
\nu_v\left(\Si_{\cW}\right) = 0
\]
for all vertices $v\in\D$.
\end{lem}
\begin{proof}
Since the measures $\nu_u$ and $\nu_v$ are mutually absolutely continuous for
all vertices $u,v\in\D$, it is sufficient to show that
$\nu_v\left(\Si_{\cW}\right) =
0$ for some vertex $v\in\D$.

Fix a vertex $v=a_{0,0}\in\cW$ and label $\cW$ with respect to this vertex as
described in~\S\ref{periodic W,S and R}. Given any apartment $\cA\supset\cW$,
a labelling of $\cA$ can be found which is compatible with this labelling
of $\cW$.
We consider the two distinct cases that arise.

\par\noindent{\em Case 1: The boundary points $\w\in\Si_{\cW}$ which have
representative sectors $\cS_v(\w)$ with one sector wall contained in $\cW$.}

All such boundary points are either contained in $\Om_v^{a_{i,i}}$ or
$\Om_v^{a_{-i,-i}}$ for arbitrarily large $i\in\NN$. Hence the set of such
boundary
points has trivial measure.

\par\noindent{\em Case 2: The boundary points $\w\in\Si_{\cW}$ which have
representative sectors $\cS_v(\w)$ neither of whose sector walls is contained in
$\cW$.}

Geometrically we are in the situation depicted in Figure~\ref{second case}

\refstepcounter{picture}
\begin{figure}[htbp]
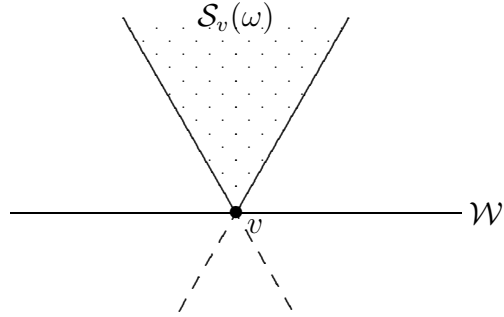
\label{second case}
{}\hfill
\beginpicture
\setcoordinatesystem units <0.5cm,0.866cm>    
\setplotarea x from -7 to 7, y from -2 to 4         
\put {$\bullet$} at 0 0
\put {$v$}             [t l]  at    0.3   -0.1
\put {$\cS_v(\w)$}      at    0 3
\put {$\cW$}         [l]    at    6.2  0
\putrule from -6    0     to  6    0
\setlinear
\plot -3  3   0 0  3  3 /
\setdashes
\plot -1.5 -1.5   0 0   1.5 -1.5 /
\vshade -3 3 3 <z,z,z,z> 0 0 3 <z,z,z,z> 3 3 3 /
\endpicture
\hfill{}
\caption{Geometric interpretation of second case.}
\end{figure}

The convex hull of $\cS=\cS_v(\w)$ and $\cW$ in any apartment containing them
both determines a root $\cR$ which is unique in the sense that it is
independent of
the particular apartment chosen. In $\cR$, denote by $\cS_1=\cS_v(\w_1)$ the
unique sector based at $v$ such that $a_{i,i}\in\cS_1$ for all $i\in\NN$. The
boundary point $\w_1$ determined by $\cS_1$ satisfies the conditions for the
first case considered and hence the set of all boundary points obtained in this
manner belong to a set of measure zero.

There is a bijective map $r_{\Om} : \w \mapsto\w_1$ which is implemented in any
apartment $\cA$ containing $\cR$ via a reflection in the wall separating
$\cS$ from
$\cS_1$ (see Figure~\ref{reflection}).

\refstepcounter{picture}
\begin{figure}[htbp]
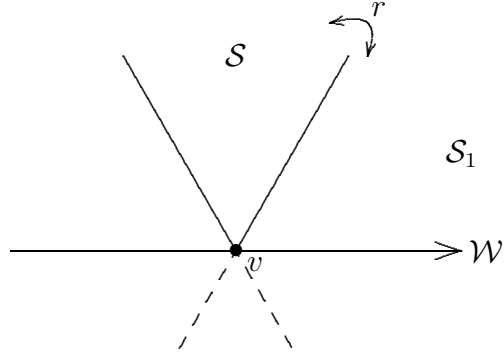
\label{reflection}
{}\hfill
\beginpicture
\setcoordinatesystem units <0.5cm,0.866cm>    
\setplotarea x from -7 to 7, y from -2 to 4         
\put {$\bullet$} at 0 0
\put {$v$}             [t l]  at    0.3   -0.1
\put {$\cS$}                 at    0 3
\put {$\cS_1$}             at    6 1.5
\put {$\cW$}         [l]    at    6.2  0
\put {$r$}              [bl]  at    3.6 3.6
\putrule from -6    0     to  6    0
\setlinear
\plot -3  3  0 0  3  3 /
\arrow <10pt> [.3,.67] from 5.5 0 to 6 0
\arrow <5pt> [.3,.67] from 2.6 3.51  to 2.5 3.5
\arrow <5pt> [.3,.67] from 3.545 3.1 to 3.5 3
\setquadratic
\plot 2.5 3.5  3.5 3.5   3.5 3 /
\setdashes
\setlinear
\plot -1.5 -1.5   0 0  1.5 -1.5 /
\endpicture
\hfill{}
\caption{Implementation of correspondence in $\cA$.}
\end{figure}

We note that
\[
\nu_v\left(\Om_v^u\right) = \nu_v\left(\Om_v^{r(u)}\right)
\]
for all vertices $u\in\cA$ since $r$ is a distance-preserving transformation of
$\cA$. Hence $r_{\Om}$ is a measure-preserving map.

Since the set of all boundary points covered by the first case has trivial
measure,
the same must therefore be true of the set of boundary points covered by the
second case.
\end{proof}

We note that there are only a countable number of periodic walls in $\D$.
This is
because $\G$  is countable so that we have a countable number of choices for the
base vertex $a_{0,0}\in\cW$ and a countable number of choices for the elements
$g\in\G$ such that $g\cW=\cW$. We can therefore deduce the following
corollary of
Lemma~\ref{measure of apartments with a periodic wall}.

\begin{cor}\label{Sigma has measure zero}
Let $\Si\subseteq\Om$ be the set of boundary points which have representative
sectors in an apartment containing a periodic wall. Then
\[
\nu_v\left(\Si\right) = 0
\]
for all vertices $v\in\D$.
\end{cor}

We are now in a position to prove the following useful result.

\begin{lem}\label{Pi has measure zero}
Given any vertex $v\in\D$, $\nu_v(\Pi)=0$.
\end{lem}
\begin{proof}
Since $\La$ is countable and points have measure zero, $\nu_v(\La)=0$ for
all vertices $v\in\D$.

Now consider the set $\Pi\setminus\La$. This consists of periodic limit points
whose periodic representatives are not rigidly periodic. By
Lemma~\ref{complete description of periodic sectors} any
$\w\in\Pi\setminus\La$ has a periodic representative $\cS(\w)$ which determines
a periodic root $\cR$ which contains periodic walls. Hence every element
$\w\in\Pi\setminus\La$ has a representative in an apartment containing a
periodic wall. Thus $\Pi\setminus\La\subseteq\Si$ and therefore
$\nu_v\left(\Pi\setminus\La\right)=0$ for all vertices $v\in\D$
by Corollary~\ref{Sigma has measure zero}.

We deduce that $\nu_v(\Pi)=0$ as required.
\end{proof}

Proposition~\ref{only boundary points are stable} and Lemma~\ref{Pi has measure
zero} have some immediate consequences.

\begin{prop}\label{measure-theoretic freeness}
The action of $\G$ on $\Om$ is measure-theoretically free, i.e.
\[
\nu_v\left(\{ \w\in\Om : g\w= \w\}\right) = 0
\]
for all vertices $v\in\D$ and elements $g\in\G\setminus\{ e\}$.
\end{prop}
\begin{proof}
As a result of Proposition~\ref{only boundary points are stable},
\[
\{ \w\in\Om : g\w= \w\} \subseteq \Pi
\]
for any $g\in\G\setminus\{ e\}$. Hence
\[
\nu_v\left(\{ \w\in\Om : g\w= \w\}\right) \leq \nu_v(\Pi)
\]
and the result follows from Lemma~\ref{Pi has measure zero}.
\end{proof}

\begin{prop}\label{ergodicity of G-action}
The action of $\G$ on $\Om$ is ergodic.
\end{prop}
\begin{proof}
Since the measures $\nu_v$ are mutually absolutely continuous, it is enough to
prove this for the measure $\nu=\nu_e$. We must therefore show that any
$\G$-invariant function $f\in L^\infty(\Om,\nu)$ is constant a.e.

By the Lebesgue differentiation theorem (see~\cite[Chapter~8,\S36]{Zaa}
and~\cite[Theorem~8.8]{Rudin}), we have for almost all $\w_0\in\Om$
\begin{equation}\label{integral equation}
\lim_{n\rightarrow\infty} \frac{1}{\nu(\Om_e^{v_n})}
\int_{\Om_e^{v_n}} \left| f(\w)-f(\w_0)\right| d\nu(\w) =0
\end{equation}
where $v_n\in\cS_e(\w)\cap V_e^{n,n}$, so that $\Om_e^{v_n}$ is a contracting
sequence of neighbourhoods of $\w_0$.

Choose such a point $\w_0$ and write $f(\w_0)=\alpha$. For each $n\in\NN$, write
$v_n=g_nz_n$  where $g_n\in V_e^{n-1,n-1}$ and $z_n\in V_e^{1,1}$ (see
Figure~\ref{vn and gn}).

\refstepcounter{picture}
\begin{figure}[htbp]
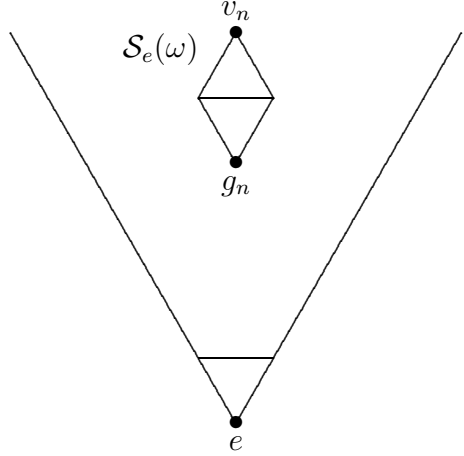
\label{vn and gn}
{}\hfill
\beginpicture
\setcoordinatesystem units <0.5cm,0.866cm>    
\setplotarea x from -6 to 6, y from -1 to 6         
\put {$\bullet$} at 0 0
\put {$\bullet$} at 0 4
\put {$\bullet$} at 0 6
\put {$e$}                 [t]  at    0  -0.2
\put {$g_n$}             [t] at    0  3.8
\put {$v_n$}             [b] at    0  6.2
\put {$\cS_e(\w)$} [b]  at   -2  5.5
\putrule from -1    1     to  1    1
\putrule from -1    5     to  1    5
\setlinear
\plot -6  6   0 0   6 6 /
\plot    0 4   -1 5   0 6   1 5   0 4 /
\endpicture
\hfill{}
\caption{Relative positions of $v_n$ and $g_n$.}
\end{figure}
Note that $\Om_{g_n}^{v_n}=\Om_{e}^{v_n}$. Now
\begin{eqnarray}
\frac{1}{\nu(\Om_e^{z_n})}\int_{\Om_e^{z_n}} \left| f(\w)-\alpha\right| d\nu(\w)
&=&
\frac{1}{\nu(\Om_e^{z_n})}
\int_{g_n\Om_e^{z_n}} \left| f(g_n^{-1}\w)-\alpha\right| d\nu(g_n^{-1}\w)
\nonumber \\
&=&
\frac{1}{\nu(\Om_e^{z_n})}
\int_{\Om_e^{v_n}} \left| f(\w)-\alpha\right| d\nu_{g_n}(\w)
\label{integral with alpha}
\end{eqnarray}
since $g_n\Om_e^{z_n} = \Om_{g_n}^{g_n z_n} = \Om_{g_n}^{v_n} = \Om_e^{v_n}$,
and $f$ is $\G$-invariant.

Furthermore, it folows from~\cite[Lemma~2.5]{CMS} that, for sufficiently large
$n\in\NN$, the Radon Nikodym derivative $\frac{d\nu_{g_n}}{d\nu}$ is constant
almost everywhere on $\Om_e^{v_n}$. Moreover,
\[
\frac{d\nu_{g_n}}{d\nu}(\w) =
\frac{\nu_{g_n}\left(\Om_e^{v_n}\right)}{\nu\left(\Om_e^{v_n}\right)} =
\frac{\nu\left(g_n^{-1}\Om_e^{v_n}\right)}{\nu\left(\Om_e^{v_n}\right)} =
\frac{\nu\left(\Om_e^{z_n}\right)}{\nu\left(\Om_e^{v_n}\right)} .
\]
The expression labelled~\eqref{integral with alpha} therefore becomes
\begin{equation}\label{altered integral with alpha}
\frac{1}{\nu(\Om_e^{v_n})}
\int_{\Om_e^{v_n}} \left| f(\w)-\alpha\right| d\nu(\w)
\end{equation}

Therefore, by~\eqref{integral equation} and~\eqref{integral with alpha},
\[
\frac{1}{\nu(\Om_e^{z_n})}\int_{\Om_e^{z_n}} \left| f(\w)-\alpha\right| d\nu(\w)
\rightarrow 0 \text{ as } n\rightarrow\infty .
\]
Now $z_n$ lies in the finite set $V_e^{1,1}$ for all $n\in\NN$, so we can
choose a
subsequence $n_k$ such that $z_{n_k}=z$ for all $k$. Then
\[
\int_{\Om_e^{z}} \left| f(\w)-\alpha\right| d\nu(\w) =0.
\]
Therefore $f(\w)=\alpha$ for almost all $\w\in\Om_e^z$.

It follows from~\cite[Lemma~3.1]{RS}  that for each possible base chamber $C=\{
e,a^{-1},b\}$ there exists an element $g\in\G$ such that
$g\Om_e^C\subseteq\Om_e^z$. Hence $f(\w)=\alpha$ a.e.($\nu$) on
$g\Om_e^C$. By the $\G$-invariance of
$f$, $f(\w)=\alpha$ a.e.($\nu_g$) on $\Om_e^C$. Since the measures
$\nu,\nu_g$ are
mutually absolutely continuous, $f(\w)=\alpha$ a.e.($\nu$) on each set $\Om_e^C$
and hence on all of $\Om$.
\end{proof}

\begin{remark}
The above proof is based on that for the case of the free group which is
contained
in~\cite[Proposition~3.9]{PS}.
\end{remark}

\begin{cor}\label{factor}
The von Neumann algebra $L^\infty(\Om)\rtimes\G$ is a factor.
\end{cor}
\begin{proof}
This follows directly from Proposition~\ref{measure-theoretic freeness},
Proposition~\ref{ergodicity of G-action} and standard results in the theory
of von
Neumann algebras (see~\cite[Proposition~4.1.15]{Sunder} for example).
\end{proof}

\begin{prop}\label{amenability of G-action}
The action of $\G$ on $\Om$ is amenable.
\end{prop}
\begin{proof}
It is sufficient, by~\cite[Th\'eor\`eme~3.3b)]{Delaroche}, to show that
there exists
a sequence $\{ f_i\}_{i\in\NN}$ of real-valued functions $L^\infty(\G\times\Om)$
such that
\begin{enumerate}
\item $\displaystyle\sum_{g\in\G}\left|f_i(g,\w)\right|^2=1$ for all
$\w\in\Om$ and
$i\in\NN$
\item
$\displaystyle\lim_i\sum_{g\in\G}\overline{f_i(g,\w)}f_i(h^{-1}g,h^{-1}\w)=1$
ultraweakly in $L^\infty(\Om)$ for each $h\in\G$.
\end{enumerate}

For each $\w\in\Om$, let $f_i(\cdot,\w)$ be the characteristic function of
\[
\{g\in S_e(\w) : |g|\leq i-1\},
\]
normalized so that the first condition holds. Thus
\[
f_i(g,\w)=\begin{cases} \left(\frac{i(i+1)}{2}\right)^{-\frac{1}{2}} &
\text{ if $g\in
S_e(\w)$ and
$|g|\leq i-1$},
\\ 0 & \text{ otherwise} \end{cases}
\]
and, for each $\w\in\Om$, there are exactly $\left(\frac{i(i+1)}{2}\right)$
elements
$g\in\G$ for which
$f_i(g,\w)\neq 0$. It was proved in~\cite[Proposition~4.2.1]{RS} that this
sequence
of functions satisfies a stronger version of these conditions, since we actually
have uniform convergence in the second condition.
\end{proof}

\begin{cor}\label{hyperfinite}
The von Neumann algebra $L^\infty(\Om)\rtimes\G$ is a hyperfinite factor.
\end{cor}

\section{Some Algebras $L^\infty(\Om)\rtimes G$}

We refer to~\cite[Definition~4.1.2]{Sunder} for the definition of the crossed
product von~Neumann algebra $L^\infty(\Om,\nu)\rtimes G$ associated with the
action of a group $G$ on a measure space $(\Om,\nu)$.

In this section we investigate some von~Neumann algebras which arise in this
manner where $(\Om,\nu)$ is the measure space described in \S\ref{measure space
omega}, paying particular attention to $L^\infty(\Om)\rtimes\G$.  We
begin by recalling some classical definitions.

\begin{defn}
Given a group $\G$ acting on a measure space $\Om$, we define the {\bf full
group}, $[\G]$, of $\G$ by
\[
[\G]=\left\{ T\in\aut(\Om) : T \w\in \G\w \text{ for almost every
}\w\in\Om\right\}.
\]
The set $[\G]_0$ of measure preserving maps in $[\G]$ is then given by
\[
[\G]_0 =\left\{ T\in [\G] : T{\scriptstyle\circ}\nu =\nu\right\}
\]
\end{defn}

\begin{defn}
Let $G$ be a countable group of automorphisms of the measure space $(\Om,\nu)$.
Following W.~Krieger, define the {\bf ratio set} $r(G)$ to be the subset of
$[0,\infty)$ such that if $\lambda\geq 0$ then $\lambda\in r(G)$ if and
only if for
every $\epsilon>0$ and Borel set $\cE$ with $\nu(\cE)>0$, there exists a
$g\in G$
and a Borel set $\cF$ such that $\nu(\cF)>0$, $\cF\cup g\cF\subseteq\cE$ and
\[
\left| \frac{d\nu{\scriptstyle\circ}g}{d\nu}(\w) -\lambda \right| <\epsilon
\]
for all $\w\in\cF$.
\end{defn}

\begin{remark}\label{full group determines ratio set}
The ratio set $r(G)$ depends only on the quasi-equivalence class of the measure
$\nu$, see~\cite[\S I-3, Lemma 14]{HO}. It also depends only on the full
group in the
sense that
\[
[H]=[G] \Rightarrow r(H)=r(G).
\]
\end{remark}

Let $\D$ be an arbitrary triangle building of order $q$ with base vertex $e$ and
write $\nu=\nu_e$.

\begin{prop}\label{ratio set of G}
Let $G$ be a countable subgroup of $\traut\subseteq\aut(\Om)$ . Suppose there
exist an element $g\in G$ such that $d(ge,e)=1$ and a subgroup $K$ of $[G]_0$
whose action on $\Om$ is ergodic. Then
\[
r(G)=\left\{ q^{2n}: n\in\ZZ\}\cup\{ 0\right\}.
\]
\end{prop}
\begin{proof}
By Remark~\ref{full group determines ratio set}, it is
sufficient to prove the statement for some group $H$ such that $[H]=[G]$. In
particular, since $[G]= [\left< G, K\right> ]$ for any subgroup $K$ of
$[G]_0$, we may
assume without loss of generality that $K\leq G$.

By~\cite[Lemma~2.5]{CMS}, for each $g\in G$, $\w\in\Om$ we have
\[
\frac{d\nu{\scriptstyle\circ}g}{d\nu}(\w)\in\left\{ q^{2n}: n\in\ZZ\}\cup\{
0\right\}.
\]
Since $G$ acts ergodically on $\Om$, $r(G)\setminus\{ 0\}$ is a group. It is
therefore enough to show that $q^2\in r(G)$.  Write $x=ge$ and note that
$\nu_x=\nu{\scriptstyle\circ}g^{-1}$. By ~\cite[Lemmas~2.2 and~2.5]{CMS} we
have
\begin{equation}\label{q2 derivative}
\frac{d\nu_x}{d\nu}(\w) =q^2, \text{ for all } \w\in\Om_e^x.
\end{equation}
Let $\cE\subseteq\Om$ be a Borel set with $\nu(\cE)>0$. By the ergodicity
of $K$,
there exist $k_1,k_2\in K$ such that the set
\[
\cF=\{\w\in \cE: k_1\w\in\Om_e^x \text{ and } k_2g^{-1}k_1\w\in \cE\}
\]
has positive measure.

Finally, let $t=k_2g^{-1}k_1\in G$. By construction, $\cF\cup t\cF\subseteq
\cE$.
Moreover, since $K$ is measure-preserving,
\[
\frac{d\nu{\scriptstyle\circ}t}{d\nu}(\w)=
\frac{d\nu{\scriptstyle\circ}g^{-1}}{d\nu}(k_1\w)=
\frac{d\nu_x}{d\nu}(k_1\w) =q^2 \text{ for all } \w\in \cF
\]
by~\eqref{q2 derivative}, since $k_1\in\Om_e^x$. This proves $q^2\in r(G)$, as
required.
\end{proof}

By definition (e.g.~\cite[\S I-3]{HO}), the conclusion of
Proposition~\ref{ratio set of
G} says that the action of $G$ is of type~$\tqs$.

\begin{cor}\label{factor if free and ergodic}
If, in addition to the hypotheses for Proposition~\ref{ratio set of G}, the
action of
$G$ is free, then $L^\infty(\Om)\rtimes G$ is a factor of type~$\tqs$.
\end{cor}
\begin{proof}
Having determined the ratio set, this is immediate
from~\cite[Corollaire~3.3.4]{Connes73}.
\end{proof}

Thus, if we can find a countable subgroup $K\leq [G]_0$ whose action on $\Om$ is
ergodic we will have shown that $L^\infty(\Om)\rtimes G$ is a factor of
type~$\tqs$. To this end, we prove the following sufficiency condition for
ergodicity.

\begin{lem}\label{transitive K is ergodic}
Let $K$ be group which acts on $\Om$. If $K$ acts transitively on the
collection of
sets $\left\{ \Om_e^x :  x\in V_e^{m,n} \right\}$ for every pair
$(m,n)\in(\NN\times\NN)$, then $K$ acts ergodically on $\Om$.
\end{lem}
\begin{proof}
Suppose now that $X_0\subseteq\Om$ is a Borel set which is invariant under $K$
and such that $\nu(X_0)>0$. We show $\nu(\Om\setminus X_0)=0$, thus establishing
the ergodicity of the action.

Define a new measure $\mu$ by $\mu(X)=\nu(X\cap X_0)$ for each Borel set
$X\subseteq\Om$. Now, for each $g\in K$,
\begin{eqnarray*}
\mu(gX)&=&\nu(gX\cap X_0) = \nu(X\cap g^{-1}X_0) \\
&\leq & \nu(X\cap X_0) + \nu(X\cap (g^{-1}X_0\setminus X_0)) \\
&=& \nu(X\cap X_0) \\
&=& \mu(X),
\end{eqnarray*}
and therefore $\mu$ is $K$-invariant.

For each $u,v\in V_e^{m,n}$ there exists a $g\in K$ such that $g\Om_e^u=\Om_e^v$
by transitivity. Thus $\mu(\Om_e^u) = \mu(\Om_e^v)$. Since
$\Om$ is the union of $N_{m,n}$ disjoint sets $\Om_e^u$, $u\in V_e^{m,n}$
each of
which has equal measure we deduce that
\[
\mu(\Om_e^u)=\frac{c}{N_{m,n}}, \text{ for each } u\in V_e^{m,n},
\]
where $c=\mu(X_0)=\nu(X_0)>0$. Thus $\mu(\Om_e^u)=c\nu(\Om_e^u)$ for every
vertex $u\in \D$.

Since the sets $\Om_e^u$ generate the Borel $\s$-algebra, we deduce that
$\mu(X)=c\nu(X)$ for each Borel set $X$. Therefore
\begin{eqnarray*}
\nu(\Om\setminus X_0) &=& c^{-1}\mu(\Om\setminus X_0) \\
&=& c^{-1}\nu((\Om\setminus X_0) \cap X_0) = 0 ,
\end{eqnarray*}
thus proving ergodicity.
\end{proof}

The following result shows that we can assume the group $K\leq\aut(\Om)$ is
countable without any loss of generality.

\begin{lem}
Assume that $K\leq\aut(\Om)$ acts transitively on the collection of sets
\[
\left\{ \Om_e^x :  x\in V_e^{m,n} \right\}
\]
for every pair $(m,n)\in\NN\times\NN$. Then there is a countable subgroup $K_0$
of $K$ which also acts transitively on the collection of sets $\left\{
\Om_e^x :  x\in
V_e^{m,n}\right\}$ for every pair $(m,n)\in\NN\times\NN$.
\end{lem}
\begin{proof}
For each pair $v,w\in V_e^{m,n}$, there exists an element $k\in K$ such that
$k\Om_e^u=\Om_e^v$. Choose one such element $k\in K$ and label it $k_{v,w}$.
Since $V_e^{m,n}$ is finite, there are a finite number of elements
$k_{v,w}\in K$ for
each $V_e^{m,n}$. There are countably many sets $V_e^{m,n}$, so the set
$\{ k_{v,w} : v,w\in V_e^{m,n} , m,n\geq 0\}$  is countable. Hence the group
\[
K_0=\left< k_{v,w} : v,w\in V_e^{m,n} , m,n\geq 0 \right>\leq K
\]
is countable and satisfies the required condition.
\end{proof}

\subsection{The Algebra $L^\infty(\Om)\rtimes\G$}

We aim to construct a subgroup $K\leq[\G]_0$ which acts ergodically on $\Om$
and use Proposition~\ref{ratio set of G} and Lemma~\ref{transitive K is
ergodic} to
prove that $L^\infty(\Om)\rtimes\G$ is a factor of type~$\tqs$. We begin with a
technical remark.

\begin{lem}\label{opposite chambers}
Given a fixed vertex $x\in\D_\cT$ and a fixed chamber $C$ with $x\in C$,
there are
precisely $q^3$ chambers $D$ with the property that $x\in D$ and for every
$\w\in\Om_x^D$, $S_x(\w)$ is opposite $S_x(\wa)$ for every $\wa\in\Om_x^C$.
\end{lem}
\begin{proof}
Since the neighbours of any vertex can be identified with the
projective plane of order $q$ introduced in~\S\ref{intro to triangle groups} we
may use properties of projective planes to prove this result.

The chambers incident on $x$ correspond to point line pairs $\{ p,l\}$.
Suppose $C$
corresponds to $\{ p_1,l_1\}$. Then the chambers $D$ will correspond to
point line
pairs $\{p_2,l_2\}$ for which there exists an incidence diagram of the type
shown in
Figure~\ref{x on opposite chambers}.
\begin{figure}[htbp]
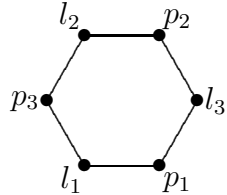
\label{x on opposite chambers}
{}\hfill
\beginpicture
\setcoordinatesystem units <0.5cm,0.866cm>    
\setplotarea x from -4 to 4, y from 4 to 8         
\put {$\bullet$} at -1 5
\put {$\bullet$} at   1 5
\put {$\bullet$} at   -2 6
\put {$\bullet$} at   2 6
\put {$\bullet$} at    -1 7
\put {$\bullet$} at   1 7
\put {$l_1$}  [t r] at   -1.05  5
\put {$p_1$}  [t l] at   1.1  4.9
\put {$l_2$}  [b r] at   -1.1  7.1
\put {$p_2$}  [b l] at   1.1  7.1
\put {$p_3$}  [r] at   -2.2  6
\put {$l_3$}   [l] at   2.2  6
\putrule from -1    5     to  1    5
\putrule from -1    7     to  1    7
\setlinear
\plot  -1 5   -2 6  -1 7  /
\plot    1 7  2 6  1 5 /
\endpicture
\hfill{}
\caption{Incidence diagram for point line pairs $\{ p_2,l_2\}$ corresponding to
chambers $D$}
\end{figure}
for some point $p_3$ and some line $l_3$.

Having fixed the pair $\{ l_1,p_1\}$ we can choose any point $p_2\not\in
l_1$, so
there are $(q^2+q+1)-(q+1)=q^2$  choices for $p_2$. Having chosen $p_2$,
the line
$l_3$ is then uniquely determined.

The number of possible choices for the line $l_2$ is then determined by the
number
of possible choices for the point $p_3$. The only restrictions on $p_3$ are that
$p_3\in l_1$ but $p_3\neq p_1$. So there are $(q+1)-1=q$ choices for $p_3$, and
hence for $l_2$.

Hence there are $q^2q=q^3$ possible pairs $\{ p_2,l_2\}$ satisfying the
necessary
conditions.
\end{proof}

We can now construct the ergodic subgroup of $[\G]_0$ provided $q\geq 3$.

\begin{prop}\label{groups K for Gamma}
If $q\geq 3$ there is a countable ergodic group $K\leq\aut(\Om)$ such that
$K\leq[\G]_0$.
\end{prop}
\begin{proof}
Let $x,y\in V_e^{m,n}$. We construct a measure preserving automorphism $k_{x,y}$
of $\Om$ such that
\begin{enumerate}
\item $k_{x,y}$ is almost everywhere a bijection from $\Om^x$ onto $\Om^y$,
\item $k_{x,y}$ is the identity on $\Om\setminus(\Om^x\cup\Om^y)$.
\end{enumerate}
It then follows from Lemma~\ref{transitive K is ergodic} that the group
\[
K=\left< k_{x,y} : \{x,y\}\subseteq V_e^{m,n}, m,n\in\NN \right>
\]
acts ergodically on $\Om$ and the construction will show explicitly that
$K\leq [\G]_0$.

Recall from \S\ref{measure space omega} that the Borel $\s$-algebra is
generated by sets of the form $\Om_v^u$ and that such a set is the disjoint
union
of open sets of the from $\Om_v^C$ for some chamber $C$ with $v\in C$. It is
therefore enough to show that, for every $x,y\in V_e^{m,n}$ and chambers
$C$ and $D$ satisfying $x\in C$, $y\in D$, $\Om_x^C=\Om_e^C$ and
$\Om_y^D=\Om_e^D$,  there exists an almost everywhere bijection
\[
k:\Om_x^{C} \longrightarrow \Om_y^{D}
\]
which is measure preserving and is pointwise approximable by the action of $\G$
almost everywhere on $\Om_x^C$.

As depicted in Figure~\ref{x,C,x1 and C1} the set $\Om_{x}^{C}$ is a
disjoint union
of sets of the form $\Om_{x_1}^{C_1}$ where $x_1\in V_e^{m+1,n+1}$ and
$\Om_{x_1}^{C_1}=\Om_e^{C_1}$.
\refstepcounter{picture}
\begin{figure}[htbp]
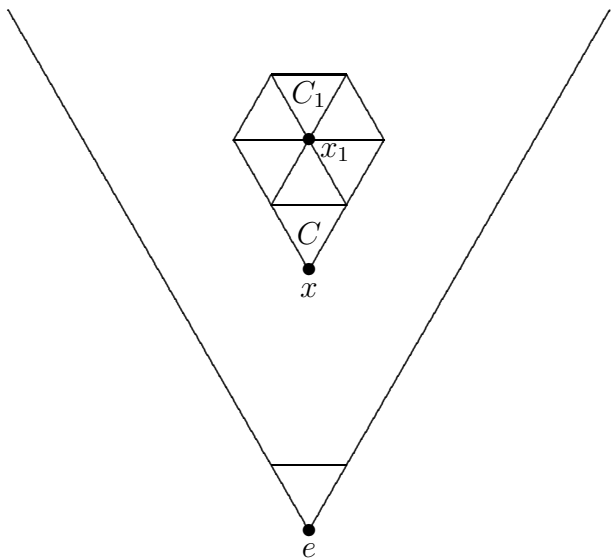
\label{x,C,x1 and C1}
{}\hfill
\beginpicture
\setcoordinatesystem units <0.5cm,0.866cm>    
\setplotarea x from -8 to 8, y from -1 to 8         
\put {$\bullet$} at 0 0
\put {$\bullet$} at 0 4
\put {$\bullet$} at 0 6
\put {$e$}           [t]  at    0  -0.2
\put {$x$}            [t] at    0  3.8
\put {$x_1$}       [t l] at   0.3  5.95
\put {$C$}           [b]  at    0  4.4
\put {$C_1$}       [b]  at    0  6.5
\putrule from -1    1     to  1    1
\putrule from -1    5     to  1    5
\putrule from -2    6     to  2    6
\putrule from -1    7     to  1    7
\setlinear
\plot -8  8   0 0   8 8 /
\plot    0 4   -1 5   0 6   1 5   0 4 /
\plot  -1 5   -2 6  -1 7  0 6  1 7  2 6  1 5 /
\endpicture
\hfill{}
\caption{Relative positions of $x$, $x_1$, $C$ and $C_1$.}
\end{figure}
Similarly the set $\Om_{y}^{D}$ is a disjoint union of sets of the form
$\Om_{y_1}^{D_1}$ where $y_1\in V_e^{m+1,n+1}$ and
$\Om_{y_1}^{D_1}=\Om_e^{D_1}$.

Fix two such vertices $x_1,y_1$. By Lemma~\ref{opposite chambers} there are
$q^3$ possible choices for each of the chambers $C_1$ and $D_1$.

Now $x_1^{-1}C_1$ and $y_1^{-1}D_1$ are each one of the $\alpha=(q+1)(q^2+q+1)$
chambers based at $e$. Therefore, since $2q^3 > \alpha$ for $q\geq 3$, we can
choose chambers $C_1$ and $D_1$ such that  $x_1^{-1}C_1= y_1^{-1}D_1$,
i.e.\ such
that $y_1x_1^{-1}C_1= D_1$. Define $k$ from $\Om_{x_1}^{C_1}$ onto
$\Om_{y_1}^{D_1}$ by
\[
k\w= y_1x_1^{-1}\w \text{ for }\w\in \Om_{x_1}^{C_1}.
\]
Thus the action of $k$ is measure preserving and pointwise approximable by $\G$
on $\Om_{x_1}^{C_1}$. Therefore $k$  remains undefined on a proportion
$\frac{\alpha-1}{\alpha}$ of $\Om_{x}^{C}$.

Now repeat the process on each of the pairs of sets
$\Om_{x_1}^{C_1},\Om_{y_1}^{D_1}$ for which $k$ has not been defined. As before,
$k$ can be defined everywhere except on a proportion
$\frac{\alpha-1}{\alpha}$ of
each such set, and $k$ can therefore be defined everywhere except on a
proportion
$\left( \frac{\alpha-1}{\alpha}\right)^2$ of the original set $\Om_{x}^{C}$.

Continuing in this manner we find that at the $n$th step, $k$ has been defined
everywhere except on a proportion $\left( \frac{\alpha-1}{\alpha}\right)^n$ of
$\Om_{x}^{C}$. Since
\[
\left( \frac{\alpha-1}{\alpha}\right)^n \rightarrow 0 \text{ as } n\rightarrow\infty,
\]
$k$ is defined almost everywhere on $\Om_{x}^{C}$ and satisfies the required
properties.
\end{proof}

The above considerations lead us to the following conclusion.

\begin{thm}\label{main theorem for Gamma}
The von Neumann algebra $L^\infty(\Om)\rtimes\G$  is a hyperfinite factor.

Moreover, if $q\geq 3$ the action of $\G$ on $\Om$ is of type~$\tqs$ and so
$L^\infty(\Om)\rtimes\G$ is the hyperfinite factor of type~$\tqs$.
\end{thm}
\begin{proof}
The von Neumann algebra $L^\infty(\Om)\rtimes\G$ is a hyperfinite factor by
Corollary~{factor} and Corollary~{hyperfinite}.

If $q\geq 3$, then Propositions~\ref{transitive K is ergodic}
and~\ref{groups K for
Gamma} prove that the factor is of type~$\tqs$.
\end{proof}

\begin{remark}
We believe that the result is also true when $q=2$, but the proof of
Proposition~\ref{groups K for Gamma} in that case appears to be harder.
\end{remark}

\subsection{Algebras From Classical Groups}\label{section 5}

We now restrict our attention to the triangle buildings associated to certain
linear groups. Henceforth, let $\cG=\pgl(3,\FF)$ for $\FF$ a local field with a
discrete valuation and let $\cK=\pgl(3,\cO)$ where $\cO$ is the valuation
ring of
$\FF$. Let $\D$ be the triangle building associated to
$\cG$~\cite[Chapter~9]{Ronan} and $\Om$ its boundary. Then $\cK$ satisfies the
conditions of Lemma~\ref{transitive K is ergodic} by the remark
following~\cite[Proposition~4.2]{CMS}.

\pagebreak[2]
\begin{prop}\label{free curly G}
$\cG$ acts freely on $\Om$.
\end{prop}
\begin{proof}
Denote by $\cP\leq\cG$ the group of upper triangular matrices in $\cG$.
By~\cite[Proposition~VI.9F]{Brown}, $\Om$ is isomorphic to $\cG/\cP$ as a
topological $\cG$-space. Moreover, $\nu$ corresponds to the unique
quasi-invariant measure on $\cG/\cP$.

Note that if $\cF$ is a closed subgroup of $\cG$ then the quasi-invariant
measure
$\mu_{\cG/\cF}$ on $\cG/\cF$ has the property that $\mu_{\cG/\cF}(Y)=0$ if and
only if $\mu_{\cG}(\pi^{-1}(Y))=0$ where $\mu_{\cG}$ denotes left Haar
measure on
$\cG$ and $\pi$ denotes  the quotient map $\cG\rightarrow \cG/\cF$
(see~\cite[VII,~\S2,~Th\'eor\`eme~1]{B78}).

Let $g\in\cG\setminus\{ 0\}$. We must show that
$\nu\left(\{\w\in\Om : g\w=\w\}\right) = 0$. That is to say we must show
\[
\mu_{\cG/\cP}\{ h\cP : h^{-1}gh\in\cP\} =0,
\]
or, equivalently, that
\[
\mu_{\cG}\{ h\in\cG : h^{-1}gh\in\cP\} =0.
\]
The condition $h^{-1}gh\in\cP$ means that any representative of $h$ in
$\gl(3,\FF)$
lies in the zero set of some nonzero polynomial
$\psi\in\FF[X_1,\ldots,X_q]$. The
zero set of $\psi$ in $\FF^q$ has measure zero, relative to the usual Haar
measure~\cite[VII,~Lemme~9]{B78}. The assertion follows from the explicit
expression for the Haar measure in
$\gl(3,\FF)$~\cite[VII,~\S3~$n^o$1,~Exemple~1]{B78}.
\end{proof}

\begin{thm}\label{PGLFq factor}
\begin{enumerate}

\item Let $q$ be a prime power and let $\Om$ be the boundary of the
building associated to $\pgl\left(3,\FF_q\left( (X)\right)\right)$. Then
$L^\infty(\Om)\rtimes\pgl\left(3,\FF_q(X)\right)$ is a factor of type~$\tqs$.

\item Let $p$ be a prime and $\Om$ the boundary of the building associated to
$\pgl(3,\QQ_p)$.  Let $\cR$ denote either $\ZZ\left[\frac{1}{p}\right]$ or
$\QQ$.
Then $L^\infty(\Om)\rtimes\pgl\left(3,\cR\right)$
is a factor of type~$\tps$.
\end{enumerate}
\end{thm}
\begin{proof}
(1) We apply Proposition~\ref{ratio set of G} with
$G=\pgl\left(3,\FF_q(X)\right)$
and $K=\pgl\left(3,\FF_q[X]\right)$. The action of $G$ is free by
Proposition~\ref{free curly G}. We must check that the action of $K$ is ergodic.
Let $K'=\pgl\left(3,\FF_q\left[ [X]\right]\right)$. Then $K'$ acts
transitively on
each set $V_e^{m,n}$ by~\cite[remark following Proposition~4.2]{CMS}. Also
$K$ is
dense in $K'$ and the action of $K'$ on $\Om$ is continuous. It follows from the
definition of the topology on $\Om$ that $K$ also acts transitively on each
$V_e^{m,n}$. Thus Lemma~\ref{transitive K is ergodic} ensures that $K$ acts
ergodically on $\Om$. Since the action of $G$ is free and ergodic the crossed
product $L^\infty(\Om)\rtimes G$ is therefore a factor.

In order to verify the remaining hypothesis of Proposition~\ref{ratio set
of G} ,
suppose that the base vertex $e$ of $\D$ is represented by the lattice class of
$<e_1,e_2,e_3>$. Then the lattice class of $<Xe_1,e_2,e_3>$ represents a
neighbouring vertex, so the element $g\in G$ represented by the matrix
$\left(\begin{array}{ccc}X&0&0\\0&1&0\\0&0&1\end{array}\right)$ satisfies
$d(ge,e)=1$.

The result now follows from Proposition~\ref{ratio set of G}.

(2)  The  proof is analogous to that given for the first statement but with
$K=\pgl(3,\ZZ)$ and $G=\pgl\left(3,\cR\right)$.
\end{proof}

\begin{remarks}
\begin{enumerate}
\item  We note that since the action of $\pgl(3,\ZZ)$ in
Theorem~\ref{PGLFq factor} is measure-preserving, it follows that
$L^\infty(\Om)\rtimes\pgl\left(3,\ZZ\right)$ is a factor of
type~${\text{\rm{II}}}_{1}$.

\item The real analogue of Theorem~\ref{PGLFq factor} is known. Take
$G=\pgl(n,\RR)$ with $n>1$ and $\Om= G/P$ where $P$ is the group of upper
triangular matrices in $G$. It then follows from a result of D.~Sullivan and
R.~Zimmer (see~\cite{Spa}) that
$L^\infty(\Om)\rtimes\pgl\left(n,\QQ\right)$ is a
factor of type~${\text{\rm{III}}}_{1}$.

\item The building associated to $\pgl\left(2,\FF_q\left(
(X)\right)\right)$ is a
homogeneous tree of degree $q+1$. In this case the corresponding factor is of
type~${\text{\rm{III}}}_{1/q}$. The building associated to
$\pgl\left(n,\FF_q\left(
(X)\right)\right)$ is a simplicial complex of rank $n-1$ and degree $q+1$.
Preliminary investigations indicate  that a similar construction will yield a
factor of type~${\text{\rm{III}}}_{1/q}$ if $n$ is odd and~$\tqs$ if $n$ is
even.

\end{enumerate}
\end{remarks}

\end{document}